\documentclass{amsart}

\usepackage{mathrsfs}
\usepackage{stmaryrd}
\usepackage{enumerate}
\usepackage{tikz-cd}
\usepackage{aliascnt}
\usepackage[colorlinks=true, linkcolor=black, citecolor=magenta]{hyperref}

\newcommand{\Bibkeyhack}[3]{}

\newtheorem{theorem}{Theorem}

\newaliascnt{lemma}{theorem}
\newtheorem{lemma}[lemma]{Lemma}
\aliascntresetthe{lemma}

\newaliascnt{corollary}{theorem}
\newtheorem{corollary}[corollary]{Corollary}
\aliascntresetthe{corollary}

\newaliascnt{proposition}{theorem}
\newtheorem{proposition}[proposition]{Proposition}
\aliascntresetthe{proposition}

\newaliascnt{remark}{theorem}
\newtheorem{remark}[remark]{Remark}
\aliascntresetthe{remark}

\theoremstyle{definition}

\newaliascnt{definition}{theorem}
\newtheorem{definition}[definition]{Definition}
\aliascntresetthe{definition}

\newtheorem*{acknowledgements}{Acknowledgements}

\newcommand{\C}{\mathbb{C}}
\newcommand{\R}{\mathbb{R}}

\newcommand{\Z}{\mathbb{Z}}

\newcommand{\bG}{\mathbb{G}}
\newcommand{\bA}{\mathbb{A}}

\newcommand{\fX}{\mathfrak{X}}
\newcommand{\fY}{\mathfrak{Y}}
\newcommand{\fA}{\mathfrak{A}}

\newcommand{\bL}{\mathbb{L}}

\newcommand{\sL}{\mathscr{L}}
\newcommand{\sG}{\mathscr{G}}

\newcommand{\sM}{\mathscr{M}}

\DeclareMathOperator{\ord}{ord}

\DeclareMathOperator{\ordjac}{ordjac}

\newcommand{\dmu}{\mathrm{d}\mu}

\DeclareMathOperator{\init}{in}

\DeclareMathOperator{\Spec}{Spec}

\DeclareMathOperator{\Trop}{Trop}

\DeclareMathOperator{\trop}{trop}

\newcommand{\Var}{\mathbf{Var}}

\newcommand{\naive}{\mathrm{naive}}

\newcommand{\cO}{\mathcal{O}}

\DeclareMathOperator{\supp}{supp}

\newcommand{\cA}{\mathcal{A}}
\newcommand{\cM}{\mathcal{M}}
\newcommand{\cB}{\mathcal{B}}
\newcommand{\Gr}{\mathrm{Gr}}

\DeclareMathOperator{\wt}{wt}

\DeclareMathOperator{\rk}{rk}

\DeclareMathOperator{\Hom}{Hom}

\title{Motivic zeta functions of hyperplane arrangements}

\author{Max Kutler and Jeremy Usatine}

\address{Max Kutler, Department of Mathematics, University of Kentucky}
\email{max.kutler@uky.edu}

\address{Jeremy Usatine, Department of Mathematics, Yale University}
\email{jeremy.usatine@yale.edu}

\begin{document}
\maketitle

\begin{abstract}
For each central essential hyperplane arrangement $\mathcal{A}$ over an algebraically closed field, let $Z_\mathcal{A}^{\hat\mu}(T)$ denote the Denef-Loeser motivic zeta function of $\mathcal{A}$. We prove a formula expressing $Z_\mathcal{A}^{\hat\mu}(T)$ in terms of the Milnor fibers of related hyperplane arrangements. We use this formula to show that the map taking each complex arrangement $\mathcal{A}$ to the Hodge-Deligne specialization of $Z_{\mathcal{A}}^{\hat\mu}(T)$ is locally constant on the realization space of any loop-free matroid. We also prove a combinatorial formula expressing the motivic Igusa zeta function of $\mathcal{A}$ in terms of the characteristic polynomials of related arrangements.
\end{abstract}

\numberwithin{theorem}{section}
\numberwithin{lemma}{section}
\numberwithin{corollary}{section}
\numberwithin{proposition}{section}
\numberwithin{remark}{section}
\numberwithin{definition}{section}

\section{Introduction}

We study hyperplane arrangements and the motivic zeta functions of Denef and Loeser. Let $k$ be an algebraically closed field, and let $H_1, \dots, H_n$ be a central essential arrangement of hyperplanes in $\bA_k^d$. If $f_1, \dots, f_n$ are linear forms defining $H_1, \dots, H_n$, respectively, then we can consider the Denef-Loeser motivic zeta function $Z_f^{\hat\mu}(T)$ of $f=f_1 \cdots f_n$ and the motivic Igusa zeta function $Z_f^{\naive}(T)$ of $f$. 

Inspired by Kontsevich's theory of motivic integration \cite{Kontsevich}, Denef and Loeser defined zeta functions \cite{DenefLoeser1998, DenefLoeser2001, DenefLoeser2002} that are power series with coefficients in a Grothendieck ring of varieties. These zeta functions are related to multiple well-known invariants in singularity theory and birational geometry, and they have implications for Igusa's monodromy conjecture, a longstanding conjecture concerning the poles of Igusa's local zeta function. There has been interest in understanding these motivic zeta functions, and the closely related topological zeta function, in the case of polynomials defining hyperplane arrangements \cite{BudurSaitoYuzvinsky, BudurMustataTeitler, vanderVeer}.

In this paper, we prove a formula for $Z_f^{\hat\mu}(T)$ in terms of the classes of Milnor fibers of certain related hyperplane arrangements. We use this formula and a result in \cite{KutlerUsatine} to show that certain specializations of $Z_f^{\hat\mu}(T)$, including the Hodge-Deligne specialization, remain constant as we vary the arrangement $H_1, \dots, H_n$ within the same connected component of a matroid's realization space. We also prove a combinatorial formula for $Z_f^{\naive}(T)$ in terms of the characteristic polynomials of certain related matroids.

\subsection{Statements of main results}

Throughout this paper, $k$ will be an algebraically closed field. Before we state our results, we need to set some notation.

For each $n \in \Z_{>0}$, let $\mu_n \subset k^\times$ be the group of $n$-th roots of unity, let $K_0^{\mu_n}(\Var_k)$ be the $\mu_n$-equivariant Grothendieck ring of $k$-varieties, let $\bL \in K_0^{\mu_n}(\Var_k)$ be the class of $\bA_k^1$ with the trivial $\mu_n$-action, and let $\sM_k^{\mu_n} = K_0^{\mu_n}(\Var_k)[\bL^{-1}]$. Let $\sM_k^{\hat\mu}=\varinjlim_n \sM_k^{\mu_n}$, and let $\bL \in \sM_k^{\hat\mu}$ be the image of $\bL \in \sM_k^{\mu_n}$ for any $n$.

Let $d, n \in \Z_{>0}$, and let $\Gr_{d,n}$ be the Grassmannian of $d$-dimensional linear subspaces in $\bA_k^n = \Spec(k[x_1, \dots, x_n])$. For each $\cA \in \Gr_{d,n}(k)$, let $X_\cA$ denote the corresponding linear subspace, let $F_\cA$ be the scheme theoretic intersection of $X_\cA$ with the closed subscheme of $\bA_k^n$ defined by $(x_1 \cdots x_n -1)$, and endow $F_\cA$ with the restriction of the $\mu_n$-action on $\bA_k^n$ where each $\xi \in \mu_n$ acts by scalar multiplication. Let $Z_{\cA,k}^{\hat\mu}(T) \in \sM_k^{\hat\mu} \llbracket T \rrbracket$ be the Denef-Loeser motivic zeta function of $(x_1 \cdots x_n)|_{X_\cA}$, and let $Z_{\cA, 0}^{\hat\mu}(T) \in \sM_k^{\hat\mu} \llbracket T \rrbracket$ be the Denef-Loeser motivic zeta function of $(x_1 \cdots x_n)|_{X_\cA}$ at the origin of $\bA_k^n$.

If $X_\cA$ is not contained in a coordinate hyperplane of $\bA_k^n$, then the restrictions of the coordinates $x_i$ define a central essential hyperplane arrangement in $X_\cA$, the Milnor fiber of that hyperplane arrangement is $F_\cA$, the $\mu_n$-action on $F_\cA$ is the monodromy action, and $Z_{\cA,k}^{\hat\mu}(T)$ and $Z_{\cA,0}^{\hat\mu}(T)$ are the Denef-Loeser motivic zeta functions associated to that arrangement. Note that we are using a definition of the Milnor fiber that takes advantage of the fact that a hyperplane arrangement is defined by a homogeneous polynomial. This definition is common in the hyperplane arrangement literature, and it allows us to consider the Milnor fiber $F_\cA$ as a variety.

\begin{remark}
If $H_1, \dots, H_n$ is a central essential hyperplane arrangement in $\bA_k^d$, then any choice of linear forms defining $H_1, \dots, H_n$ gives a linear embedding of $\bA_k^d$ into $\bA_k^n$, and $H_1, \dots, H_n$ is the arrangement associated to the resulting subspace of $\bA_k^n$. Therefore, we lose no generality by considering the arrangements associated to $d$-dimensional linear subspaces in $\bA_k^n$.
\end{remark}

Let $\cM$ be a rank $d$ loop-free matroid on $\{1, \dots, n\}$, let $\Trop(\cM) \subset \R^n$ be the Bergman fan of $\cM$, and let $\Gr_\cM \subset \Gr_{d,n}$ be the locus parametrizing linear subspaces whose associated hyperplane arrangements have combinatorial type $\cM$. For any $w \in \Trop(\cM)$, there exists a rank $d$ loop-free matroid $\cM_w$ on $\{1, \dots, n\}$ such that for all $\cA \in \Gr_\cM(k)$, the initial degeneration $\init_w (X_\cA \cap \bG_{m,k}^n)$ is equal to $X_{\cA_w} \cap \bG_{m,k}^n$ for some unique $\cA_w \in \Gr_{\cM_w}(k)$. We refer to Section \ref*{linearsubspacesandmatroidssubsection} for the definition of $\cM_w$. Let $\cB(\cM)$ be the set of bases in $\cM$, and set
\[
	\wt_{\cM}: \R^n \to \R: (w_1, \dots, w_n) \mapsto \max_{B \in \cB(\cM)} \sum_{i \in B} w_i.
\]

In this paper, we will prove the following formulas that express the motivic zeta functions $Z^{\hat\mu}_{\cA, k}(T)$ and $Z^{\hat\mu}_{\cA,0}(T)$ in terms of classes of the Milnor fibers $F_{\cA_w}$.

\begin{theorem}
\label{hyperplanearrangementDLzetaformula}
Let $\cA \in \Gr_\cM(k)$. Then
\[
	Z^{\hat\mu}_{\cA, k}(T) = \sum_{w \in \Trop(\cM) \cap (\Z_{\geq 0}^n \setminus \{0\})} [F_{\cA_w}, \hat\mu] \bL^{-d-\wt_\cM(w)}(T, \dots, T)^w \in \sM_k^{\hat\mu}\llbracket T \rrbracket,
\]
and
\[
	Z^{\hat\mu}_{\cA,0}(T) = \sum_{w \in \Trop(\cM) \cap \Z_{> 0}^n} [F_{\cA_w}, \hat\mu] \bL^{-d-\wt_\cM(w)}(T, \dots, T)^w \in \sM_k^{\hat\mu}\llbracket T \rrbracket.
\]
\end{theorem}

In the course of proving \autoref*{hyperplanearrangementDLzetaformula}, we prove \autoref*{zetafunctionschon} and \autoref*{tropicalzetaformulahomogeneous}, which give formulas for motivic zeta functions when certain tropical hypotheses are satisfied. We think of \autoref*{zetafunctionschon} and \autoref*{tropicalzetaformulahomogeneous} as being in the spirit of the formulas for zeta functions of so-called Newton non-degenerate hypersurfaces \cite{DenefHoornaert, Guibert, BoriesVeys, BultotNicaise}. To prove \autoref*{zetafunctionschon} and \autoref*{tropicalzetaformulahomogeneous}, we use certain $k \llbracket \pi \rrbracket$-schemes whose special fibers are the initial degenerations that arise in tropical geometry. These $k \llbracket \pi \rrbracket$-schemes have played an essential role in much of tropical geometry. See for example \cite{Gubler}. We also use Sebag's \cite{Sebag} theory of motivic integration for Greenberg schemes, which are non-constant coefficient versions of arc schemes. For our proofs to account for the $\hat\mu$-action, we use Hartmann's \cite{Hartmann} equivariant version of Sebag's motivic integration.

\autoref*{hyperplanearrangementDLzetaformula} allows us to use results about additive invariants of the Milnor fibers $F_{\cA_w}$ to obtain results about specializations of the Denef-Loeser motivic zeta functions. To state such an application, we first define some terminology that can apply to additive invariants. Let $\Z[\bL]$ be the polynomial ring over the symbol $\bL$, and endow $\sM_k^{\hat\mu}$ with the $\Z[\bL]$-algebra structure given by $\bL \mapsto \bL$.

\begin{definition}
Let $P$ be a $\Z[\bL]$-module, and let $\nu: \sM_k^{\hat\mu} \to P$ be a $\Z[\bL]$-module morphism. We say that $\nu$ is \emph{constant on smooth projective families with $\mu_n$-action} if the following always holds.
\begin{itemize}
\item If $S$ is a connected separated finite type $k$-scheme with trivial $\mu_n$-action and $X \to S$ is a $\mu_n$-equivariant smooth projective morphism from a scheme $X$ with $\mu_n$-action, then the map $S(k) \to P: s \mapsto \nu[X_s, \hat\mu]$ is constant, where $X_s$ denotes the fiber of $X \to S$ over $s$.
\end{itemize}

\end{definition}

\begin{remark}
If $k = \C$ and $\mathrm{HD}: \sM_k^{\hat\mu} \to \Z[u^{\pm 1},v^{\pm 1}]$ is the morphism that sends the class of each variety to its Hodge-Deligne polynomial, then $\mathrm{HD}$ is constant on smooth projective families with $\mu_n$-action.
\end{remark}

Note that if $w \in \Trop(\cM)$ and $\cA_1, \cA_2 \in \Gr_{\cM}(k)$ are in the same connected component of $\Gr_\cM$, then $(\cA_1)_w, (\cA_2)_w \in \Gr_{\cM_w}(k)$ are in the same connected component of $\Gr_{\cM_w}$. See for example \cite[Fact 2.4]{KutlerUsatine}. Therefore the following theorem is a direct consequence of \autoref*{hyperplanearrangementDLzetaformula} and \cite[Theorem 1.4]{KutlerUsatine}.

\begin{theorem}
\label{specializationinvarianceinconnectedcomponent}
Let $P$ be a torsion-free $\Z[\bL]$-module, let $\nu: \sM_k^{\hat\mu} \to P$ be a $\Z[\bL]$-module morphism that is constant on smooth projective families with $\mu_n$-action, and assume that the characteristic of $k$ does not divide $n$.

If $\cA_1, \cA_2 \in \Gr_\cM(k)$ are in the same connected component of $\Gr_\cM$, then
\[
	\nu(Z_{\cA_1, k}^{\hat\mu}(T)) = \nu(Z_{\cA_2, k}^{\hat\mu}(T)) \in P\llbracket T \rrbracket,
\]
and
\[
	\nu(Z_{\cA_1, 0}^{\hat\mu}(T)) = \nu(Z_{\cA_2, 0}^{\hat\mu}(T)) \in P\llbracket T \rrbracket.
\]

\end{theorem}

\begin{remark}
In the statement of \autoref*{specializationinvarianceinconnectedcomponent}, by $\nu$ applied to a power series, we mean the power series obtained by applying $\nu$ to each coefficient.
\end{remark}

In particular, \autoref*{specializationinvarianceinconnectedcomponent} implies that the Hodge-Deligne specialization of the Denef-Loeser motivic zeta function remains constant as we vary the linear subspace within the same connected component of $\Gr_\cM$. There has been much interest in understanding how invariants of hyperplane arrangements, particularly those invariants arising in singularity theory, vary as the arrangements vary with fixed combinatorial type. For example, a major open conjecture predicts that when $k = \C$, the Betti numbers of a hyperplane arrangement's Milnor fiber depend only on combinatorial type, i.e., they depend only on the matroid. Budur and Saito proved that a related invariant, the Hodge spectrum, depends only on the combinatorial type \cite{BudurSaito}. Randell proved that the diffeomorphism type, and thus Betti numbers, of the Milnor fiber is constant in smooth families of hyperplane arrangements with fixed combinatorial type \cite{Randell}. See \cite{Suciu} for a survey on such questions. Our perspective on \autoref*{specializationinvarianceinconnectedcomponent} is in the context of that literature, and we hope it illustrates the use of \autoref*{hyperplanearrangementDLzetaformula} in answering related questions.

Our final main result consists of combinatorial formulas for the motivic Igusa zeta functions of a hyperplane arrangement. It is well known that the motivic Igusa zeta functions are combinatorial invariants. For example, one can see this by using De Concini and Procesi's wonderful models \cite{DeConciniProcesi} and Denef and Loeser's formula for the motivic Igusa zeta function in terms of a log resolution \cite[Corollary 3.3.2]{DenefLoeser2001}. Regardless, we believe it is worth stating \autoref*{hyperplanearrangementIgusazetaformula} below, as it follows from the methods of this paper with little extra effort, and because we are not aware of these particular formulas having appeared in the literature.

Let $K_0(\Var_k)$ be the Grothendieck ring of $k$-varieties, let $\bL \in K_0(\Var_k)$ be the class of $\bA_k^1$, and let $\sM_k = K_0(\Var_k)[\bL^{-1}]$. For each $\cA \in \Gr_{d,n}(k)$, let $Z_{\cA,k}^{\naive}(T) \in \sM_k \llbracket T \rrbracket$ be the motivic Igusa zeta function of $(x_1 \cdots x_n)|_{X_\cA}$, and let $Z_{\cA, 0}^{\naive}(T) \in \sM_k \llbracket T \rrbracket$ be the motivic Igusa zeta function of $(x_1 \cdots x_n)|_{X_\cA}$ at the origin of $\bA_k^n$.

\begin{theorem}
\label{hyperplanearrangementIgusazetaformula}
Let $\cA \in \Gr_\cM(k)$. Then
\[
	Z^{\naive}_{\cA,k}(T) = \sum_{w \in \Trop(\cM) \cap \Z_{\geq 0}^n} \chi_{\cM_w}(\bL) \bL^{-d-\wt_\cM(w)}(T, \dots, T)^w \in \sM_k\llbracket T \rrbracket,
\]
and 
\[
	Z^{\naive}_{\cA,0}(T) = \sum_{w \in \Trop(\cM) \cap \Z_{> 0}^n} \chi_{\cM_w}(\bL) \bL^{-d-\wt_\cM(w)}(T, \dots, T)^w \in \sM_k\llbracket T \rrbracket,
\]
where $\chi_{\cM_w}(\bL) \in \sM_k$ is the characteristic polynomial of $\cM_w$ evaluated at $\bL$.
\end{theorem}

\begin{acknowledgements}
We would like to acknowledge useful discussions with Dori Bejleri, Daniel Corey, Netanel Friedenberg, Dave Jensen, Kalina Mincheva, Sam Payne, and Dhruv Ranganathan. The second named author was supported by NSF Grant DMS-1702428 and a Graduate Research Fellowship from the NSF.
\end{acknowledgements}

\section{Preliminaries}
\label{preliminariessection}

In this section, we will set some notation and recall facts about the equivariant Grothendieck ring of varieties, the motivic zeta functions of Denef and Loeser, Hartmann's equivariant motivic integration, and linear subspaces and matroids.

\subsection{The equivariant Grothendieck ring of varieties}

Suppose $X$ is a separated finite type scheme over $k$. We will let $K_0(\Var_X)$ denote the Grothendieck ring of varieties over $X$, we will let $\bL \in K_0(\Var_X)$ denote the class of $\bA_k^1 \times_k X$, and for each separated finite type $X$-scheme $Y$, we will let $[Y/X] \in K_0(\Var_X)$ denote the class of $Y$. We will let $\sM_X$ denote the ring obtained by inverting $\bL$ in $K_0(\Var_X)$, and by slight abuse of notation, we will let $\bL, [Y/X] \in \sM_X$ denote the images of $\bL, [Y/X]$, respectively, in $\sM_X$. 

We will let $K_0(\Var_k)$ and $\sM_k$ denote $K_0(\Var_{\Spec(k)})$ and $\sM_{\Spec(k)}$, respectively, and for each separated finite type $k$-scheme $Y$, we will let $[Y] = [Y/\Spec(k)]$ in both $K_0(\Var_k)$ and $\sM_k$.

Suppose $G$ is a finite abelian group. An action of $G$ on a scheme is said to be \emph{good} if each orbit is contained in an affine open subscheme. For example, any $G$-action on any quasiprojective $k$-scheme is good. Suppose $X$ is a separated finite type $k$-scheme with a good $G$-action. We will let $K_0^G(\Var_X)$ denote the $G$-equivariant Grothendieck ring of varieties over $X$. For the precise definition of $K_0^G(\Var_X)$, we refer to \cite[Definition 4.1]{Hartmann}. We will let $\bL \in K_0^G(\Var_X)$ denote the class of $\bA_k^1 \times_k X$ with the action induced by the trivial $G$-action on $\bA_k^1$ and the given $G$-action on $X$, and for each separated finite type $X$-scheme $Y$ with good $G$-action making the structure morphism $G$-equivariant, we will let $[Y/X,G] \in K_0^G(\Var_X)$ denote the class of $Y$ with its given $G$-action. We will let $\sM_X^G$ denote the ring obtained by inverting $\bL$ in $K_0^G(\Var_X)$, and by slight abuse of notation, we will let $\bL, [Y/X,G] \in \sM_X^G$ denote the images of $\bL, [Y/X, G]$, respectively, in $\sM_X^G$.

If $X$ is a separated finite type $k$-scheme with no specified $G$-action and we refer to $K_0^G(\Var_X)$ or $\sM_X^G$, then we are considering $X$ with the trivial $G$-action. We will let $K_0^G(\Var_k)$ and $\sM_k^G$ denote $K_0^G(\Var_{\Spec(k)})$ and $\sM_{\Spec(k)}^G$, respectively, and for each separated finite type $k$-scheme $Y$ with good $G$-action making the structure morphism $G$-equivariant, we will let $[Y, G] = [Y/\Spec(k), G]$ in both $K_0^G(\Var_k)$ and $\sM_k$.

For each $\ell \in \Z_{>0}$, we will let $\mu_\ell \subset k^\times$ denote the group of $\ell$-th roots of unity. 

\begin{remark}
We will only consider $\mu_\ell$ as a finite group, so when the characteristic of $k$ divides $\ell$, we will not consider the non-reduced scheme structure of $\mu_\ell$. 
\end{remark}

For each $\ell, m \in \Z_{>0}$, there is a morphism $\mu_{\ell m} \to \mu_\ell: \xi \mapsto \xi^m$. Suppose that $X$ is a separated finite type scheme over $k$. Then for each $\ell, m \in \Z_{>0}$, the morphism $\mu_{\ell m} \to \mu_\ell$ induces ring morphisms $K_0^{\mu_\ell}(\Var_X) \to K_0^{\mu_{\ell m}}(\Var_X)$ and $\sM_X^{\mu_\ell} \to \sM_X^{\mu_{\ell m}}$. We will let $K_0^{\hat\mu}(\Var_X) = \varinjlim_{\ell}K_0^{\mu_\ell}(\Var_X)$ and $\sM_X^{\hat\mu} = \varinjlim_\ell \sM_X^{\mu_\ell}$. We will let $\bL \in K_0^{\hat\mu}(\Var_X)$ denote the image of $\bL \in K_0^{\mu_\ell}(\Var_X)$ for any $\ell \in \Z_{>0}$, and similarly we will let $\bL \in \sM_X^{\hat\mu}$ denote the image of $\bL \in \sM_X^{\mu_\ell}$ for any $\ell \in \Z_{>0}$. For each $\ell \in \Z_{>0}$ and each separated finite type $X$-scheme $Y$ with good $\mu_\ell$-action making the structure morphism $\mu_\ell$-equivariant, we will let $[Y/X, \hat\mu] \in K_0^{\hat\mu}(\Var_X)$ denote the image of $[Y/X, \mu_\ell] \in K_0^{\mu_\ell}(\Var_X)$, and we will similarly let $[Y/X, \hat\mu] \in \sM_X^{\hat\mu}$ denote the image of $[Y/X, \mu_\ell] \in \sM_X^{\mu_\ell}$.

We will let $K_0^{\hat\mu}(\Var_k)$ and $\sM_k^{\hat\mu}$ denote $K_0^{\hat\mu}(\Var_{\Spec(k)})$ and $\sM_{\Spec(k)}^{\hat\mu}$, respectively, and for each $\ell \in \Z_{>0}$ and each separated finite type $k$-scheme $Y$ with good $\mu_\ell$-action making the structure morphism $\mu_\ell$-equivaraint, we will let $[Y, \hat\mu] = [Y/\Spec(k), \hat\mu]$ in both $K_0^{\hat\mu}(\Var_k)$ and $\sM_k^{\hat\mu}$.

\subsection{The motivic zeta functions of Denef and Loeser}

Let $X$ be a smooth, pure dimensional, separated, finite type $k$-scheme. For each $\ell \in \Z_{\geq 0}$, we will let $\sL_\ell(X)$ denote $\ell$-th jet scheme of $X$, and for each $m \geq \ell$, we will let $\theta^m_\ell: \sL_m(X) \to \sL_\ell(X)$ denote the truncation morphism. We will let $\sL(X) = \varprojlim_\ell \sL_\ell(X)$ denote the arc scheme of $X$, and for each $\ell \in \Z_{\geq 0}$, we will let $\theta_\ell: \sL(X) \to \sL_\ell(X)$ denote the canonical morphism. The following is a special case of a theorem of Bhatt's \cite[Theorem 1.1]{Bhatt}.

\begin{theorem}[Bhatt]
The $k$-scheme $\sL(X)$ represents the functor taking each $k$-algebra $A$ to $\Hom_k(\Spec(A\llbracket \pi \rrbracket), X)$, and under this identification, each morphism $\theta_\ell: \sL(X) \to \sL_\ell(X)$ is the truncation morphism.
\end{theorem}

A subset of $\sL(X)$ is called a \emph{cylinder} if it is the preimage, under $\theta_\ell$, of a constructible subset of $\sL_\ell(X)$ for some $\ell \in \Z_{\geq 0}$. We will let $\mu_X$ denote the motivic measure on $\sL(X)$, which assigns a motivic volume in $\sM_X$ to each cylinder.

Suppose $f$ is a regular function on $X$. If $x \in \sL(X)$ has residue field $k(x)$, then it corresponds to a $k$-morphism $\psi_x: \Spec(k(x)\llbracket \pi \rrbracket) \to X$, and we will let $f(x)$ denote $f(\psi_x) \in k(x)\llbracket \pi \rrbracket$. For each $x \in \sL(X)$, the \emph{order} of $f$ at $x$ will refer to the order of $\pi$ in the power series $f(x)$, and the \emph{angular component} of $f$ at $x$ will refer to the leading coefficient of the power series $f(x)$. We will let $\ord_f: \sL(X) \to \Z_{\geq 0} \cup \{\infty\}$ denote the function taking each $x \in \sL(X)$ to the order of $f$ at $x$. We will let $Z^\naive_f(T) \in \sM_X\llbracket T\rrbracket$ denote the motivic Igusa zeta function of $f$. Then
\[
	Z^\naive_f(T) = \sum_{\ell \in \Z_{\geq 0}} \mu_X(\ord_f^{-1}(\ell)) T^\ell \in \sM_X \llbracket T \rrbracket.
\]

\begin{remark}
In the literature, the motivic Igusa zeta function is sometimes referred to as the naive zeta function of Denef and Loeser.
\end{remark}

We will let $Z_f^{\hat\mu}(T) \in \sM_X^{\hat\mu}\llbracket T\rrbracket$ denote the Denef-Loeser motivic zeta function of $f$. We briefly recall the definition of $Z_f^{\hat\mu}(T)$. The constant term of $Z_f^{\hat\mu}(T)$ is equal to $0$. Let $\ell \in \Z_{>0}$, and let $Y_{\ell,1}$ be the closed subscheme of $\sL_\ell(X)$ where $f$ is equal to $\pi^\ell$. For each $k$-algebra $A$, there is a $\mu_\ell$-action on $A\llbracket \pi \rrbracket$ given by $\pi \mapsto \xi \pi$ for each $\xi \in \mu_\ell$, and these actions induce a $\mu_\ell$-action on $\sL_\ell(X)$ making $Y_{\ell, 1}$ invariant. Note also that the truncation morphism $\theta^\ell_0: \sL_\ell(X) \to X$ restricts to a $\mu_\ell$-equivariant morphism $Y_{\ell, 1} \to X$. Then the coefficient of $T^\ell$ in $Z_f^{\hat\mu}(T)$ is defined to be equal to $[Y_{\ell, 1}/X, \hat\mu]\bL^{-(\ell+1)\dim X} \in \sM_X^{\hat\mu}$.

\begin{remark}
Denef and Loeser defined versions of these zeta functions with coefficients in $\sM_k$ and $\sM^{\hat\mu}_k$ \cite{DenefLoeser1998, DenefLoeser2002}, and Looijenga introduced versions with coefficients in the relative Grothendieck rings $\sM_X$ and $\sM_X^{\hat\mu}$ \cite{Looijenga}. See \cite{DenefLoeser2001} for the definitions we are using for $Z_f^{\naive}(T)$ and $Z_f^{\hat\mu}(T)$, but note that compared to those definitions, ours differ by a normalization factor of $\bL^{-\dim X}$.
\end{remark}

\subsection{Hartmann's equivariant motivic integration}

For the remainder of this paper, let $R = k\llbracket \pi \rrbracket$, the ring of power series over $k$. We will set up some notation and recall facts for Greenberg schemes and Hartmann's equivariant motivic integration \cite{Hartmann}, which is an equivariant version of Sebag's motivic integration for formal schemes \cite{Sebag}. For the non-equivariant version of this theory, we also recommend the book \cite{ChambertLoirNicaiseSebag}.

\begin{remark}
In \cite{Hartmann}, Hartmann uses formal $R$-schemes. The analogous theory for algebraic $R$-schemes, as stated here, directly follows by taking $\pi$-adic completion.
\end{remark}

Let $\fX$ be a smooth, pure relative dimensional, separated, finite type $R$-scheme. We will let $\fX_0$ denote the special fiber of $\fX$. For each $\ell \in \Z_{\geq 0}$, we will let $\sG_\ell(\fX)$ denote the $\ell$-th Greenberg scheme of $\fX$. Thus $\sG_\ell(\fX)$ represents the functor taking each $k$-algebra $A$ to $\Hom_R(\Spec(A[\pi]/(\pi^{\ell+1})), \fX)$. For each $m \geq \ell$, we will let $\theta^m_\ell: \sG_m(\fX) \to \sG_\ell(\fX)$ denote the truncation morphism. We will let $\sG(\fX) = \varprojlim_\ell \sG_\ell(\fX)$ denote the Greenberg scheme of $\fX$, and for each $\ell \in \Z_{\geq 0}$, we will let $\theta_\ell: \sG(\fX) \to \sG_\ell(\fX)$ denote the canonical morphism. As for arc schemes, the following is a special case of \cite[Theorem 1.1]{Bhatt}. See for example \cite[Chapter 4, Proposition 3.1.7]{ChambertLoirNicaiseSebag}.

\begin{theorem}[Bhatt]
The $k$-scheme $\sG(\fX)$ represents the functor taking each $k$-algebra $A$ to $\Hom_R(\Spec(A\llbracket \pi \rrbracket), \fX)$, and under this identification, each morphism $\theta_\ell: \sG(\fX) \to \sG_\ell(\fX)$ is the truncation morphism.
\end{theorem}

A subset of $\sG(\fX)$ is called a \emph{cylinder} if it is the preimage, under $\theta_\ell$, of a constructible subset of $\sG_\ell(\fX)$ for some $\ell \in \Z_{\geq 0}$. We will let $\mu_\fX$ denote the motivic measure on $\sG(\fX)$, which assigns a motivic volume in $\sM_{\fX_0}$ to each cylinder.

Suppose $f$ is a regular function on $\fX$. If $x \in \sG(\fX)$ has residue field $k(x)$, then it corresponds to an $R$-morphism $\psi_x: \Spec(k(x)\llbracket \pi \rrbracket) \to \fX$, and we will let $f(x)$ denote $f(\psi_x) \in k(x)\llbracket \pi \rrbracket$. As for arc schemes, this is used to define the \emph{order} and \emph{angular component} of $f$ at $x$ and the order function $\ord_f: \sG(\fX) \to \Z_{\geq 0} \cup \{\infty\}$.

Now suppose $G$ is a finite abelian group acting on $R$, and suppose that each element of $G$ acts on $R$ by a $\pi$-adically continuous $k$-algebra morphism. Endow $\fX$ with a good $G$-action making the structure morphism $G$-equivariant, and endow $\fX_0$ with the restriction of the $G$-action on $\fX$. The $G$-action on $\fX$ induces good $G$-actions on $\sG(\fX)$ and each $\sG_\ell(\fX)$. We refer to \cite[Section 3.2]{Hartmann} for the construction and properties of these $G$-actions on the Greenberg schemes. We will let $\mu_{\fX}^G$ denote the $G$-equivariant motivic measure on $\sG(\fX)$, which assigns a motivic volume in $\sM_{\fX_0}^G$ to each $G$-invariant cylinder in $\sG(\fX)$. We refer to \cite[Section 4.2]{Hartmann} for the definition of $\mu_{\fX}^G$.

If $A \subset \sG(\fX)$ is a $G$-invariant cylinder and $\alpha: A \to \Z$ is a function whose fibers are $G$-invariant cylinders, then the integral of $\alpha$ is defined to be
\[
	\int_A \bL^{-\alpha} \dmu_{\fX}^G = \sum_{\ell \in \Z} \mu_\fX^G(\alpha^{-1}(\ell)) \bL^{-\ell} \in \sM_{\fX_0}^G.
\]

\begin{remark}
By the quasi-compactness of the construcible topology, $\alpha$ takes finitely many values, so the above sum is well defined. See \cite[Chaper 6, Section 1.2]{ChambertLoirNicaiseSebag}.
\end{remark}

We now state the equivariant version of the motivic change of variables formula \cite[Theorem 4.18]{Hartmann}. If $h: \fY \to \fX$ is a morphism of $R$-schemes, then we let $\ordjac_h: \sG(\fY) \to \Z_{\geq 0} \cup \{\infty\}$ denote the order function of the jacobian ideal of $h$.

\begin{theorem}[Hartmann]
Suppose $\# G$ is not divisible by the characteristic of $k$. Let $\fX, \fY$ be smooth, pure relative dimensional, separated, finite type $R$-schemes with good $G$-action making the structure morphisms equivariant, and let $h: \fY \to \fX$ be a $G$-equivariant morphism that induces an open immersion on generic fibers. Let $A,B$ be $G$-invariant cylinders in $\sG(\fX), \sG(\fY)$, respectively, such that $h$ induces a bijection $B(k') \to A(k')$ for all extensions $k'$ of $k$.

If $\alpha: A \to \Z$ is a function whose fibers are $G$-invariant cylinders, then $\alpha \circ \sG(h) - \ordjac_h: B \to \Z$ is a function whose fibers are $G$-invariant cylinders, and
\[
	\int_A \bL^{-\alpha} \dmu_{\fX}^G = \int_B \bL^{-(\alpha \circ \sG(h) + \ordjac_h)} \dmu_{\fY}^G \in \sM_{\fX_0}^G.
\]
\end{theorem}

\begin{remark}
Hartmann stated the formula when $A = \sG(\fX)$ and $B = \sG(\fY)$, but the same proof works when replacing $\sG(\fX)$ and $\sG(\fY)$ with $G$-invariant cylinders. See for example the proof of the non-equivariant version in \cite{ChambertLoirNicaiseSebag}.
\end{remark}

We note that for all $\ell \in \Z_{>0}$, the characteristic of $k$ never divides $\# \mu_\ell$.

\subsection{Linear subspaces and matroids}
\label{linearsubspacesandmatroidssubsection}

Let $d, n \in \Z_{>0}$. We will let $\Gr_{d,n}$ denote the Grassmannian of $d$-dimensional linear subspaces in $\bA_k^n = \Spec(k[x_1, \dots, x_n])$. We will let $\bG_{m,k}^n = \Spec(k[x_1^{\pm 1}, \dots, x_n^{\pm 1}]) \subset \bA_k^n$ denote the complement of the coordinate hyperplanes, and we will let $V(x_1 \cdots x_n -1)$ denote the closed subscheme of $\bA_k^n$ defined by $(x_1 \cdots x_n -1)$. For each $\cA \in \Gr_{d,n}(k)$, we will let $X_\cA \hookrightarrow \bA_k^n$ denote the corresponding linear subspace. If $X_\cA$ is not contained in a coordinate hyperplane of $\bA_k^n$, then the restrictions to $X_\cA$ of the coordinates $x_i$ define a central essential hyperplane arrangement in $X_\cA$. We let $U_\cA = X_\cA \cap \bG_{m,k}^n$ and $F_\cA = X_\cA \cap V(x_1 \cdots x_n -1)$ denote this arrangement's complement and Milnor fiber, respectively, and we endow $F_\cA$ with the restriction of the $\mu_n$-action on $\bA_k^n$ where each $\xi \in \mu_n$ acts by scalar multiplication. In the context of tropical geometry, we will consider both $U_\cA$ and $F_\cA$ as closed subschemes of the algebraic torus $\bG_{m,k}^n$. We will let $Z_{\cA}^{\hat\mu}(T) \in \sM^{\hat\mu}_{X_\cA}\llbracket T \rrbracket$ and $Z_{\cA}^{\naive}(T) \in \sM_{X_\cA} \llbracket T \rrbracket$ denote the Denef-Loeser motivic zeta function and the motivic Igusa zeta function, respectively, of the restriction of $(x_1 \cdots x_n)$ to $X_\cA$. We will let $Z^{\hat\mu}_{\cA, k}(T) \in \sM_k^{\hat\mu} \llbracket T \rrbracket$ (resp. $Z^{\naive}_{\cA, k}(T) \in \sM_k \llbracket T \rrbracket$) denote the power series obtained by pushing forward each coefficient of $Z_{\cA}^{\hat\mu}(T)$ (resp. $Z_{\cA}^{\naive}(T)$) along the structure morphism of $X_\cA$. We will let $Z^{\hat\mu}_{\cA, 0}(T) \in \sM_k^{\hat\mu} \llbracket T \rrbracket$ (resp. $Z^{\naive}_{\cA, 0}(T) \in \sM_k \llbracket T \rrbracket$) denote the power series obtained by pulling back each coefficient of $Z_{\cA}^{\hat\mu}(T)$ (resp. $Z_{\cA}^{\naive}(T)$) along the inclusion of the origin into $X_\cA$.

\begin{remark}
The zeta functions $Z_{\cA,k}^{\hat\mu}(T), Z_{\cA,k}^\naive(T), Z_{\cA,0}^{\hat\mu}(T)$, and $Z_{\cA, 0}^\naive(T)$ are as denoted in the introduction of this paper.
\end{remark}

Let $\cM$ be a rank $d$ loop-free matroid on $\{1, \dots, n\}$. We will let $\chi_\cM(\bL) \in \sM_k$ denote the characteristic polynomial of $\cM$ evaluated at $\bL$, so
\[
	\chi_\cM(\bL) = \sum_{I \subset \{1, \dots, n\}} (-1)^{\# I} \bL^{d-\rk I} \in \sM_k,
\]
where $\rk I$ is the rank function of $\cM$ applied to $I$. We will let $\cB(\cM)$ denote the set of bases of $\cM$, and we will let $\wt_\cM: \R^n \to \R$ denote the function $(w_1, \dots, w_n) \mapsto \max_{B \in \cB(\cM)} \sum_{i \in B} w_i$. For each $w = (w_1, \dots, w_n) \in \R^n$, we will set
\[
	\cB(\cM_w) = \{ B \in \cB(\cM) \, | \, \sum_{i \in B} w_i = \wt_\cM(w)\}.
\]
Then $\cB(\cM_w)$ is the set of bases for a rank $d$ matroid on $\{1, \dots, n\}$, and we will let $\cM_w$ denote that matroid. We let $\Trop(\cM)$ denote the Bergman fan of $\cM$, so
\[
	\Trop(\cM) = \{w \in \R^n \, | \, \text{$\cM_w$ is loop-free}\}.
\]
We will let $\Gr_\cM \subset \Gr_{d,n}$ denote the locus parametrizing linear subspaces whose associated hyperplane arrangements have combinatorial type $\cM$. For all $\cA \in \Gr_\cM(k)$, the fact that $\cM$ is loop-free implies that $X_\cA$ is not contained in a coordinate hyperplane. Note that if $\cA \in \Gr_\cM(k)$, then $[U_\cA] = \chi_\cM(\bL) \in \sM_k$ and
\[
	\Trop(U_\cA) = \{w \in \R^n \, | \, \init_w U_\cA \neq \emptyset\} = \Trop(\cM).
\]
For each $\cA \in \Gr_\cM(k)$ and each $w \in \Trop(\cM)$, we will let $\cA_w \in \Gr_{\cM_w}(k)$ denote the unique point such that $\init_w U_\cA = U_{\cA_w}$.

Before concluding the preliminaries, we recall two propositions proved in \cite{KutlerUsatine} that will be used in Section \ref*{motiviczetafunctionsprovingmainresultsection}. If $B \in \cB(\cM)$ and $i \in \{1, \dots, n\} \setminus B$, then we will let $C(\cM, i, B)$ denote the fundamental circuit in $\cM$ of $B$ with respect to $i$, so $C(\cM, i, B)$ is the unique circuit in $\cM$ contained in $B \cup \{i\}$. For each circuit $C$ in $\cM$ and each $\cA \in \Gr_\cM(k)$, we will let $L_C^\cA \in k[x_1, \dots, x_n]$ denote a linear form in the ideal defining $X_\cA$ in $\bA_k^n$ such that the coefficient of $x_i$ in $L_C^\cA$ is nonzero if and only if $i \in C$. Such an $L_C^\cA$ exists and is unique up to scaling by a unit in $k$. Once and for all, we fix such an $L_C^\cA$ for all $C$ and $\cA$.

\begin{proposition}[Proposition 3.6 in \cite{KutlerUsatine}]
\label{gbasisforlinearsubspace}
Let $\cA \in \Gr_\cM(k)$, let $w \in \R^n$, and let $B \in \cB(\cM_w)$. Then
\[
	\{L_{C(\cM, i,B)}^\cA \, | \, i \in \{1, \dots, n\} \setminus B\} \subset k[x_1, \dots, x_n]
\]
generates the ideal of $X_\cA$ in $\bA_k^n$, and
\[
	\{ \init_w L_{C(\cM, i, B)}^\cA \, | \, i \in \{1, \dots, n\} \setminus B\} \subset k[x_1^{\pm}, \dots, x_n^{\pm}]
\]
generates the ideal of $\init_w U_{\cA}$ in $\bG_{m,k}^n$.
\end{proposition}

\begin{proposition}[Proposition 3.2 in \cite{KutlerUsatine}]
\label{wmaximalextrahasminimalweight}
Let $w = (w_1, \dots, w_n) \in \R^n$, let $B \in \cB(\cM_w)$, and let $i \in \{1, \dots, n\} \setminus B$. Then
\[
	\min_{j \in C(\cM,i,B)} w_j = w_i.
\]
\end{proposition}

For additional information on matroids and the tropical geometry of linear subspaces, we refer to \cite[Chapter 4]{MaclaganSturmfels}.

\section{Equivariant motivic integration and the motivic zeta function}

Let $\ell \in \Z_{>0}$, and throughout this section, endow $R$ with the $\mu_\ell$-action where each $\xi \in \mu_{\ell}$ acts on $R$ by the $\pi$-adically continuous $k$-morphism $\pi \mapsto \xi^{-1}\pi$.

Let $X$ be a smooth, pure dimensional, finite type, separated scheme over $k$. We will endow $\sL(X)$ and each $\sL_m(X)$ with $\mu_\ell$-actions that make the truncation morphisms $\mu_\ell$-equivariant as follows. Let $\xi \in \mu_\ell$, let $A$ be a $k$-algebra, let $\xi_{A\llbracket \pi \rrbracket}: \Spec(A\llbracket \pi \rrbracket) \to \Spec(A \llbracket \pi \rrbracket)$ be the morphism whose pullback is the $\pi$-adically continuous $A$-algebra morphism $\pi \mapsto \xi^{-1}\pi$, and let $\xi_{A[ \pi ]/(\pi^{n+1})}: \Spec(A[ \pi ]/(\pi^{n+1})) \to \Spec(A[ \pi ]/(\pi^{n+1}))$ be the morphism whose pullback is the $A$-algebra morphism $\pi \mapsto \xi^{-1}\pi$.

If $x \in \sL(X)(A)$ corresponds to a $k$-morphism 
\[
	\psi_x: \Spec(A \llbracket \pi \rrbracket) \to X,
\]
then let $\xi \cdot x \in \sL(X)(A)$ correspond to the $k$-morphism 
\[
	\psi_x \circ \xi_{A \llbracket \pi \rrbracket}^{-1}: \Spec(A \llbracket \pi \rrbracket) \to X.
\]
This action is clearly functorial in $A$, so it defines a $\mu_\ell$-action on $\sL(X)$. Similarly, if $x \in \sL_m(X)(A)$ corresponds to a $k$-morphism 
\[
	\psi_x: \Spec(A [ \pi ]/(\pi^{m+1})) \to X,
\]
then let $\xi \cdot x \in \sL_m(X)(A)$ correspond to the $k$-morphism 
\[
	\psi_x \circ \xi_{A [ \pi ]/(\pi^{m+1})}^{-1}: \Spec(A [ \pi ]/(\pi^{m+1})) \to X.
\] 
This action is also functorial in $A$, so it defines a $\mu_\ell$-action on $\sL_m(X)$. We also see that these $\mu_\ell$-actions make the truncation morphisms $\mu_\ell$-equivariant.

\begin{proposition}
\label{orderangularcomponentarcscheme}
Let $f$ be a regular function on $X$. Then $f$ has constant order on any $\mu_\ell$-orbit of $\sL(X)$. Furthermore, $f$ has constant angular component on any $\mu_\ell$-orbit of $\sL(X)$ on which $f$ has order $\ell$.
\end{proposition}

\begin{proof}
Let $\xi \in \mu_\ell$, let $\xi_{\sL(X)}: \sL(X) \to \sL(X)$ be its action on $\sL(X)$, let $x \in \sL(X)(k')$ for some extension $k'$ of $k$, let $R' = k' \llbracket \pi \rrbracket$, and let $\xi_{R'}: \Spec(R') \to \Spec(R')$ be the morphism whose pullback is the $\pi$-adically continuous $k'$-algebra morphism $\pi \mapsto \xi^{-1}\pi$. Then $x$ corresponds to a $k$-morphism
\[
	\psi_x: \Spec(R') \to X,
\]
and $\xi_{\sL(X)}(x) \in \sL(X)(k')$ corresponds to the $k$-morphism
\[
	\psi_x \circ \xi_{R'}^{-1}: \Spec(R') \to X.
\]
Write
\[
	f(x) = f(\psi_x) = \sum_{i \geq 0} a_i \pi^i \in R',
\]
where each $a_i \in k'$. Then
\[
	f(\xi_{\sL(X)}(x)) = f(\psi_x \circ \xi_{R'}^{-1}) = \sum_{i \geq 0} a_i \xi^i \pi^i \in R'.
\]
Thus the order of $f(x)$ is equal to the order of $f(\xi_{\sL(X)}(x))$, and if $f(x)$ has order $\ell$, then the fact that $\xi ^\ell = 1$ implies that the angular component of $f(x)$ is equal to the angular component of $f(\xi_{\sL(X)}(x))$. Thus we are done.
\end{proof}

Let $\fX = X \times_k \Spec(R)$ and endow $\fX$ with the $\mu_\ell$-action induced by the $\mu_\ell$-action on $R$ and the trivial $\mu_\ell$-action on $X$. Note that any open affine cover of $X$ induces an open cover of $\fX$ by $\mu_\ell$-invariant affines, so the $\mu_\ell$-action on $\fX$ is good. Composition with the projection $\fX \to X$ induces isomorphisms $\sG(\fX) \to \sL(X)$ and $\sG_m(\fX) \to \sL_m(X)$ that commute with the truncation morphisms.

\begin{proposition}
\label{equivariantisogreenbergarc}
The isomorphisms $\sG(\fX) \to \sL(X)$ and $\sG_m(\fX) \to \sL_m(X)$ are $\mu_\ell$-equivariant.
\end{proposition}

\begin{proof}
Let $m \in \Z_{\geq 0}$. It will be sufficient to show that the isomorphism $\sG_m(\fX) \to \sL_m(X)$ is $\mu_\ell$-equivariant, as we get the remainder of the proposition by taking inverse limit.

Let $\xi \in \mu_\ell$, let $\xi_\fX: \fX \to \fX$ be its action on $\fX$, and let $\xi_{\sG_m(\fX)}: \sG_m(\fX) \to \sG_m(\fX)$ be its action on $\sG_m(\fX)$.

Let $x \in \sG_m(\fX)(A)$ for some $k$-algebra $A$, and let 
\[
	\xi_{A [ \pi ]/(\pi^{m+1})}: \Spec(A [ \pi ]/(\pi^{m+1})) \to \Spec(A [ \pi ]/(\pi^{m+1}))
\]
be the morphism whose pullback is the $A$-algebra morphism $\pi \mapsto \xi^{-1} \pi$. Then $x$ corresponds to an $R$-morphism 
\[
	\psi_x: \Spec(A [ \pi ]/(\pi^{m+1})) \to \fX,
\]
and $\xi_{\sG_m(\fX)}(x) \in \sG_m(\fX)(A)$ corresponds to the $R$-morphism
\[
	\xi_{\fX} \circ \psi_x \circ \xi_{A [ \pi ]/(\pi^{m+1})}^{-1}: \Spec(A [ \pi ]/(\pi^{m+1})) \to \fX.
\]
Because $\xi_{\fX}$ is trivial on the factor $X$, we get that the composition of the above morphism with the projection $\fX \to X$ is equal to the composition of
\[
	\psi_x \circ \xi_{A [ \pi ]/(\pi^{m+1})}^{-1}: \Spec(A [ \pi ]/(\pi^{m+1})) \to \fX
\]
with the projection $\fX \to X$. Thus the proposition follows by our definition of the $\mu_\ell$-action on $\sL_m(X)$.
\end{proof}

\begin{proposition}
\label{orderangularcomponentgreenbergscheme}
Let $f$ be a regular function on $\fX$ obtained by pulling back a regular function on $X$ along the projection $\fX \to X$. Then $f$ has constant order on any $\mu_\ell$-orbit of $\sG(\fX)$. Furthermore, $f$ has constant angular component on any $\mu_\ell$-orbit of $\sG(\fX)$ on which $f$ has order $\ell$.
\end{proposition}

\begin{proof}
This follows from Propositions \ref*{orderangularcomponentarcscheme} and \ref*{equivariantisogreenbergarc}.
\end{proof}

Let $f$ be a regular function on $X$, let $Z_f^{\hat\mu}(T) \in \sM_{X}^{\hat\mu} \llbracket T \rrbracket$ denote the Denef-Loeser motivic zeta function of $f$, and let $Z_f^\naive(T) \in \sM_X \llbracket T \rrbracket$ denote the motivic Igusa zeta function of $f$. By slight abuse of notation, we will also let $f$ denote the regular function on $\fX$ obtained by pulling back $f$ along the projection $\fX \to X$.

\begin{proposition}
\label{coefficientDLfromvolume}
Let $A_{\ell,1} \subset \sG(\fX)$ be the subset of arcs where $f$ has order $\ell$ and angular component 1. Then $A_{\ell,1}$ is a $\mu_\ell$-invariant cylinder, and the coefficient of $T^\ell$ in $Z_f^{\hat\mu}(T)$ is equal to the image of $\mu_{\fX}^{\mu_\ell}(A_{\ell,1})$ in $\sM_X^{\hat\mu}$.
\end{proposition}

\begin{proof}
Let $B_{\ell,1} \subset \sL(X)$ be the subset of arcs where $f$ has order $\ell$ and angular component 1, and let $Y_{\ell, 1}$ be the closed subscheme of $\sL_\ell(X)$ consisting of jets where $f$ is equal to $\pi^\ell$. Then $\theta_\ell(B_{\ell, 1}) = Y_{\ell, 1}$. By \autoref*{orderangularcomponentarcscheme}, $B_{\ell,1}$ is a $\mu_\ell$-invariant subset of $\sL(X)$, so because $\theta_\ell$ is $\mu_{\ell}$-equiviariant, we have that $Y_{\ell, 1}$ is a $\mu_\ell$-invariant subset of $\sL_\ell(X)$. Thus we may endow $Y_{\ell, 1}$ with the $\mu_\ell$-action given by restriction of the $\mu_\ell$-action on $\sL_\ell(X)$. By the definition of $Z^{\hat\mu}_f(T)$ and the $\mu_\ell$-action on $Y_{\ell, 1}$, the coefficient of $T^\ell$ in $Z_f^{\hat\mu}(T)$ is equal to
\[
	[Y_{\ell, 1} / X, \hat\mu] \bL^{-(\ell+1)\dim X} \in \sM_X^{\hat\mu}.
\]
But by the $\mu_\ell$-equivariant isomorphisms $\sG(\fX) \to \sL(X)$ and $\sG_\ell(\fX) \to \sL_\ell(X)$, the fact that the image of $A_{\ell, 1}$ under $\sG(\fX) \to \sL(X)$ is equal to $B_{\ell, 1}$, and the fact that $\theta_\ell^{-1}(Y_{\ell, 1}) = B_{\ell,1}$, we have that $A_{\ell,1}$ is a $\mu_\ell$-invariant cylinder and
\[
	\mu_{\fX}^{\mu_\ell}(A_{\ell,1}) = [Y_{\ell, 1} / X, \mu_\ell] \bL^{-(\ell+1)\dim X} \in \sM_X^{\mu_\ell},
\]
and we are done.
\end{proof}

\begin{proposition}
\label{coefficientIgfromvolume}
Let $A_\ell \subset \sG(\fX)$ be the subset of arcs where $f$ has order order $\ell$. Then $A_\ell$ is a cylinder and the coefficient of $T^\ell$ in $Z_f^\naive(T)$ is equal to $\mu_{\fX}(A_\ell)$.
\end{proposition}

\begin{proof}
This proposition follows from the definition of $Z_f^\naive(T)$ and the fact that the isomorphism $\sG(\fX) \to \sL(X)$ is cylinder and volume preserving.
\end{proof}

\section{Actions of the roots of unity on an algebraic torus}

Let $T$ be an algebraic torus over $k$ with character lattice $M$ and co-character lattice $N$, and for each $u \in M$, let $\chi^u \in k[M]$ denote the corresponding character on $T$. In this section, we will establish some notation and facts regarding certain actions, by the roots of unity, on the closed subschemes of $T$.

\begin{definition}
Let $\ell \in \Z_{> 0}$. Let $w \in N$, and $\bG_{m,k} \to T$ be the corresponding co-character. Then we define the \emph{$(\mu_\ell, w)$-action} to be the $\mu_\ell$-action on $T$ induced by the group homomorphism $\mu_\ell \hookrightarrow \bG_{m,k} \to T$. 

For each closed subscheme $U$ of $T$ that is invariant under the $(\mu_\ell, w)$-action, we will let $U_\ell^w$ denote the scheme $U$ endowed with the $\mu_\ell$-action given by restriction of the $(\mu_\ell, w)$-action.
\end{definition}

\begin{remark}
\label{characterctionpullback}
Under the $(\mu_\ell, w)$-action, each $\xi \in \mu_\ell$ acts on $T$ with pullback
\[
	\chi^u \mapsto \xi^{\langle u, w \rangle} \chi^u.
\]
\end{remark}

\begin{proposition}
\label{initialdegenerationinvariant}
Let $\ell \in \Z_{>0}$, let $w \in N$, and let $U$ be a closed subscheme of $T$. Then the initial degeneration $\init_w U$ is a closed subscheme of $T$ that is invariant under the $(\mu_\ell, w)$-action.
\end{proposition}

\begin{proof}
Let $\xi \in \mu_\ell$, and let $\xi_T: T \to T$ be its action on $T$. It will be sufficient to show that for all $f \in k[M]$, the pullback $\xi_T^*(\init_w f)$ is contained in the ideal of $k[M]$ generated by $\init_w f$.

By definition,
\[
	\supp(\init_w f) = \{u \in \supp(f) \, | \, \langle u, w \rangle = \trop(f)(w)\},
\]
so by \autoref*{characterctionpullback},
\[
	\xi_T^*(\init_w f) = \xi^{\trop(f)(w)} \init_w f,
\]
and we are done.
\end{proof}

\begin{proposition}
\label{binomialcharacterinvariant}
Let $w \in N$, let $u \in M$ such that $\langle u, w \rangle > 0$, and let $V(\chi^u-1)$ be the closed subscheme of $T$ defined by $\chi^u-1 \in k[M]$. Then $V(\chi^u-1)$ is invariant under the $(\mu_{\langle u, w \rangle}, w)$-action.
\end{proposition}

\begin{proof}
Let $\xi \in \mu_{\langle u, w \rangle}$, and let $\xi_T: T \to T$ be its action on $T$. Then by \autoref*{characterctionpullback},
\[
	\xi_T^*(\chi^u - 1) = \xi^{\langle u, w \rangle} \chi^u - 1 = \chi^u -1,
\]
and we are done.
\end{proof}

\begin{proposition}
\label{initdegscalingcharacteraction}
Let $U$ be a closed subscheme of $T$, let $\ell \in \Z_{>0}$, let $w \in N$, and let $u \in M$ be such that $\langle u, w \rangle > 0$.

Then $U$ is invariant under the $(\mu_{\langle u, w \rangle}, w)$-action if and only if $U$ is invariant under the $(\mu_{\langle u, \ell w \rangle}, \ell w)$-action. 

Furthermore, if $U$ is invariant under the $(\mu_{\langle u, w \rangle}, w)$-action, then
\[
	[U_{\langle u, w \rangle}^w, \hat\mu] = [U_{\langle u, \ell w \rangle}^{\ell w}, \hat\mu] \in K_0^{\hat\mu}(\Var_k).
\]
\end{proposition}

\begin{proof}
Under the $(\mu_{\langle u, w \rangle}, w)$-action, each $\xi \in \mu_{\langle u, w \rangle}$ acts on $T$ with pullback
\[
	\chi^{u'} \mapsto \xi^{\langle u', w \rangle} \chi^{u'}.
\]
The homomorphism $\mu_{\langle u, \ell w \rangle} \to \mu_{\langle u, w \rangle}: \xi \mapsto \xi^\ell$ and the $(\mu_{\langle u, w \rangle}, w)$-action induce a $\mu_{\langle u, \ell w \rangle}$-action on $T$ such that each $\xi \in \mu_{\langle u, \ell w \rangle}$ acts on $T$ with pullback
\[
	\chi^{u'} \mapsto (\xi^\ell)^{\langle u', w \rangle} \chi^{u'} = \xi^{\langle u', \ell w \rangle} \chi^{u'}.
\]
We see that this action is equal to the $(\mu_{\langle u, \ell w \rangle}, \ell w)$-action. Then the surjectivity of $\mu_{\langle u, \ell w\rangle} \to \mu_{\langle u, w \rangle}$ implies that $U$ is invariant under the $(\mu_{\langle u, w \rangle}, w)$-action if and only if it is invariant under the $(\mu_{\langle u, \ell w \rangle}, \ell w)$-action. The remainder of the proposition follows from the definition of the map $K_0^{\mu_{\langle u, w \rangle}}(\Var_k) \to K_0^{\mu_{\langle u, \ell w \rangle}}(\Var_k)$.
\end{proof}

We will devote the remainder of this section to proving the following proposition.

\begin{proposition}
\label{moninitialdeg}
Let $U$ be a closed subscheme of $T$, let $u \in M$, let $V(\chi^u-1)$ be the closed subscheme of $T$ defined by $\chi^u-1 \in k[M]$, let $w \in u^\perp \cap N$, and let $v \in N$ be such that $\ell = \langle u, v \rangle > 0$ and such that $\init_w U$ is invariant under the $(\mu_\ell, v)$-action.

Then $\init_w U$ is invariant under the $(\mu_\ell, v-w)$-action, and
\[
	[(V(\chi^u-1) \cap \init_w U)_\ell^v, \mu_\ell] = [ (V(\chi^u-1) \cap \init_w U)_\ell^{v-w}, \mu_\ell] \in K_0^{\mu_\ell}(\Var_k).
\]
\end{proposition}

\begin{remark}
In the statement of \autoref*{moninitialdeg}, because $\ell = \langle u, v \rangle = \langle u, v-w \rangle$, \autoref*{binomialcharacterinvariant} implies that $V(\chi^u-1)$ is invariant under the $(\mu_\ell, v)$-action and the $(\mu_\ell, v-w)$-action, so the classes in the statement are well defined.
\end{remark}

\subsection{Proof of \autoref*{moninitialdeg}}

Let $U$ be a closed subscheme of $T$, let $u \in M$, let $V(\chi^u-1)$ be the closed subscheme of $T$ defined by $\chi^u-1 \in k[M]$, let $w \in u^\perp \cap N$, and let $v \in N$ be such that $\ell = \langle u, v \rangle > 0$ and such that $\init_w U$ is invariant under the $(\mu_\ell, v)$-action. \autoref*{moninitialdeg} is clear when $w=0$, so we assume that $w \neq 0$.

Let $\cO_w = \Spec(k[w^\perp \cap M])$, and let $T \to \cO_w$ be the algebraic group homomorphism induced by the inclusion $k[w^\perp \cap M] \to k[M]$.

\begin{lemma}
\label{shiftintoperp}
Let $f \in k[M]$. Then there exists $u' \in M$ such that $\init_w (\chi^{u'} f) \in k[w^\perp \cap M]$.
\end{lemma}

\begin{proof}
By definition, 
\[
	\supp(\init_w f) = \{u' \in \supp(f) \, | \, \langle u', w \rangle = \trop(f)(w)\}.
\]
If $f=0$, the statement is obvious. Thus we may assume that there exists $u' \in M$ such that $-u' \in \supp(\init_w f)$. Then we have that 
\[
	\init_w (\chi^{u'} f) = \chi^{u'} \init_w f \in k[w^\perp \cap M].
\]
\end{proof}

\begin{proposition}
\label{initdegpreimageprojection}
There exist closed subschemes $Y$ and $Z$ of $\cO_w$ such that $\init_w U$ is equal to the pre-image of $Y$ under the morphism $T \to \cO_w$ and $V(\chi^u-1) \cap \init_w U$ is equal to the pre-image of $Z$ under the morphism $T \to \cO_w$.
\end{proposition}

\begin{proof}
Let $f_1, \dots, f_m \in k[M]$ be such that $\init_w f_1, \dots, \init_w f_m \in k[M]$ generate the ideal defining $\init_w U$ in $T$. By \autoref*{shiftintoperp}, we can assume that $\init_w f_1, \dots, \init_w f_m \in k[w^\perp \cap M]$. Because $w \in u^\perp$, we have that $\chi^u \in k[w^\perp \cap M]$.

Thus we may let $Y$ be the closed subscheme of $\cO_w$ defined by the ideal generated by $\init_w f_1, \dots, \init_w f_m \in k[w^\perp \cap M]$, and we may let $Z$ be the closed subscheme of $\cO_w$ defined by the ideal generated by $\init_w f_1, \dots, \init_w f_m, \chi^u-1 \in k[w^\perp \cap M]$, and we are done.
\end{proof}

\begin{lemma}
\label{characterssameafterprojection}
The composition of the co-character $\bG_{m,k} \to T$ corresponding to $v$ with the morphism $T \to \cO_w$ is equal to the composition of the co-character $\bG_{m,k} \to T$ corresponding to $v-w$ with the morphism $T \to \cO_w$
\end{lemma}

\begin{proof}
The composition of the co-character $\bG_{m,k} \to T$ corresponding to $v$ with the morphism $T \to \cO_w$ corresponds to the map of lattices $w^\perp \cap M \hookrightarrow M \xrightarrow{\langle \cdot, v \rangle} \Z$, and the composition of the co-character $\bG_{m,k} \to T$ corresponding to $v-w$ with the morphism $T \to \cO_w$ corresponds to the map of lattices $w^\perp \cap M \hookrightarrow M \xrightarrow{\langle \cdot, v-w \rangle} \Z$. These are clearly the same lattice maps, so we are done.
\end{proof}

Let $T_w = \Spec(k[(\R w \cap N)^\vee])$. Any splitting of $0 \to \R w \cap N \to N \to N/(\R w \cap N) \to 0$ induces an isomorphism of algebraic groups $T \cong T_w \times_k \cO_w$ such that $T \to \cO_w$ corresponds to the projection $T_w \times_k \cO_w \to \cO_w$.

Let $\phi_1: \mu_\ell \to T$ (resp. $\phi_2: \mu_\ell \to T$) be the composition of $\mu_\ell \hookrightarrow \bG_{m,k}$ with the co-character $\bG_{m,k} \to T$ corresponding to $v$ (resp. $v-w$).

Let $\varphi_1: \mu_\ell \to \cO_w$ (resp. $\varphi_2: \mu_\ell \to \cO_w$) be the composition of $\phi_1$ (resp. $\phi_2$) with $T \to \cO_w$.

\begin{lemma}
\label{diagonalactionequalfactor}
We have that $\varphi_1 = \varphi_2$.
\end{lemma}

\begin{proof}
This follows directly from \autoref*{characterssameafterprojection}.
\end{proof}

Let $\psi_1: \mu_\ell \to T_w$ (resp. $\psi_2: \mu_\ell \to T_w$) be the composition of $\phi_1$ (resp. $\phi_2$) with the projection $T \cong T_w \times_k \cO_w \to T_w$.

\begin{remark}
\label{characteractiondiagonal}
We see that under the identification $T \cong T_w \times_k \cO_w$, the $(\mu_\ell, v)$-action (resp. $(\mu_\ell, v-w)$-action) is the diagonal action defined by the action on $\cO_w$ induced by $\varphi_1$ (resp. $\varphi_2$) and the action on $T_w$ induced by $\psi_1$ (resp. $\psi_2$).
\end{remark}

We now prove the first part of \autoref*{moninitialdeg}.

\begin{proposition}
We have that $\init_w U$ is invariant under the $(\mu_\ell, v-w)$-action.
\end{proposition}

\begin{proof}
By \autoref*{initdegpreimageprojection}, there exists a closed subscheme $Y$ of $\cO_w$ such that $\init_w U$ is equal to the pre-image of $Y$ under the morphism $T \to \cO_w$. Then under the identification $T \cong T_w \times_k \cO_w$, we have that
\[
	\init_w U = T_w \times_k Y.
\]
Because $\init_w U$ is invariant under the $(\mu_\ell, v)$-action, \autoref*{characteractiondiagonal} implies that $Y$ is invariant under the $\mu_\ell$-action on $\cO_w$ induced by $\varphi_1$. By \autoref*{diagonalactionequalfactor}, $Y$ is invariant under the $\mu_\ell$-action on $\cO_w$ induced by $\varphi_2$, and by \autoref*{characteractiondiagonal}, this implies that $\init_w U$ is invariant under the $(\mu_\ell, v-w)$-action.
\end{proof}

Before we complete the proof of \autoref*{moninitialdeg}, we make the following observation, which follows from \cite[Lemma 7.1]{KutlerUsatine} and the fact that $\dim T_w = 1$.

\begin{remark}
\label{trivialclassinvarianttorus}
The class in $K_0^{\mu_\ell}(\Var_k)$ of $T_w$ with the $\mu_\ell$-action induced by $\psi_1$ (resp. $\psi_2$) is equal to $\bL-1$.
\end{remark}

We now complete the proof of \autoref*{moninitialdeg}.

\begin{proposition}
We have that
\[
	[(V(\chi^u-1) \cap \init_w U)_\ell^v, \mu_\ell] = [ (V(\chi^u-1) \cap \init_w U)_\ell^{v-w}, \mu_\ell] \in K_0^{\mu_\ell}(\Var_k).
\]
\end{proposition}

\begin{proof}
By \autoref*{initdegpreimageprojection}, there exists a closed subscheme $Z$ of $\cO_w$ such that $V(\chi^u-1) \cap \init_w U$ is equal to the pre-image of $Z$ under the morphism $T \to \cO_w$. Then under the identification $T \cong T_w \times_k \cO_w$, we have that
\[
	V(\chi^u-1) \cap \init_w U = T_w \times_k Z.
\]
Because $V(\chi^u-1) \cap \init_w U$ is invariant under the $(\mu_\ell, v)$-action, \autoref*{characteractiondiagonal} implies that $Z$ is invariant under the $\mu_\ell$-action on $\cO_w$ induced by $\varphi_1$.

Now endow $Z$ with the $\mu_\ell$-action given by restriction of the $\mu_\ell$-action on $\cO_w$ induced by $\varphi_1$, which by \autoref*{diagonalactionequalfactor} is the same as the $\mu_\ell$-action given by restriction of the $\mu_\ell$-action on $\cO_w$ induced by $\varphi_2$. Then by Remarks \ref*{characteractiondiagonal} and \ref*{trivialclassinvarianttorus},
\begin{align*}
	[(V(\chi^u-1) \cap \init_w U)_\ell^v, \mu_\ell] &= (\bL-1)[Z, \mu_\ell]\\
	&= [ (V(\chi^u-1) \cap \init_w U)_\ell^{v-w}, \mu_\ell].
\end{align*}
\end{proof}

\section{Motivic zeta functions and smooth initial degenerations}
\label{motiviczetafunctionsforschonvarieties}

Let $n \in \Z_{>0}$, let $\bA_k^n = \Spec(k[x_1, \dots, x_n])$, and let $\bG_{m,k}^n = \Spec(k[x_1^{\pm 1}, \dots, x_n^{\pm 1}])$. Let $X$ be a smooth pure dimension $d$ closed subscheme of $\bA_k^n$ such that $U = X \cap \bG_{m,k}^n$ is nonempty and such that for all $w \in \Trop(U) \cap \Z_{\geq 0}^n$, the initial degeneration $\init_w U$ is smooth and there exist $f_1, \dots, f_{n-d} \in k[x_1, \dots, x_n]$ that generate the ideal of $X$ in $\bA_k^n$ such that $\init_w f_1, \dots, \init_w f_{n-d} \in k[x_1^{\pm 1}, \dots, x_n^{\pm 1}]$ generate the ideal of $\init_w U$ in $\bG_{m,k}^n$.

Let $u \in \Z_{>0}^n$, let $Z^{\hat\mu}_{X,u}(T) \in \sM_X^{\hat\mu}\llbracket T \rrbracket$ be the Denef-Loeser motivic zeta function of the restriction $(x_1, \dots, x_n)^u|_X$, and let $Z^{\naive}_{X,u}(T) \in \sM_X\llbracket T \rrbracket$ be the motivic Igusa zeta function of $(x_1, \dots, x_n)^u|_X$.

To state \autoref*{zetafunctionschon} below, we will need the following proposition, which will also be proved in this section.

\begin{proposition}
\label{initialdegenerationXschemestructure}
Let $w = (w_1, \dots, w_n) \in \Trop(U) \cap \Z_{\geq 0}^n$, and let $\varphi: \bG_{m,k}^n \to \bA_k^n$ be the morphism whose pullback is given by $x_i \mapsto 0^{w_i}x_i$. Then the restriction of $\varphi$ to $\init_w U$ factors through $X$, and if $w \neq 0$, the induced map $(\init_w U)^w_{u \cdot w} \to X$ is $\mu_{u \cdot w}$-equivariant with respect to the trivial $\mu_{u \cdot w}$-action on $X$.
\end{proposition}

In this section, we will prove the following theorem and its corollary.

\begin{theorem}
\label{zetafunctionschon}
Let $V_u$ be the subscheme of $\bG_{m,k}^n$ defined by $(x_1, \dots, x_n)^u - 1$. For any $w \in \Trop(U) \cap \Z_{\geq 0}^n$, endow the initial degeneration $\init_w U$ and the intersection $V_u \cap \init_w U$ with the $X$-scheme structure given by \autoref*{initialdegenerationXschemestructure}.

Then there exists a function $\ordjac: \Trop(U) \cap \Z_{\geq 0}^n \to \Z$ that satisfies the following.

\begin{enumerate}[(a)]

\item If $w = (w_1, \dots, w_n) \in \Trop(U) \cap \Z_{\geq 0}^n$ and $f_1, \dots, f_{n-d} \in k[x_1, \dots, x_n]$ are a generating set for the ideal of $X$ such that $\init_w f_1, \dots, \init_w f_{n-d} \in k[x_1^{\pm 1}, \dots, x_n^{\pm 1}]$ generate the ideal of $\init_w U$, then
\[
	\ordjac(w) = w_1 + \dots + w_n - (\trop(f_1)(w) + \dots + \trop(f_{n-d})(w)) \in \Z.
\]

\item We have that
\[
	Z^{\hat\mu}_{X,u}(T) = \sum_{w \in \Trop(U) \cap (\Z_{\geq 0}^n \setminus \{0\})} [(V_u \cap \init_w U)^w_{u \cdot w} / X, \hat\mu] \bL^{-d-\ordjac(w)} T^{u \cdot w},
\]
and
\[
	Z^\naive_{X,u}(T) = \sum_{w \in \Trop(U) \cap \Z_{\geq 0}^n} [\init_w U / X] \bL^{-d-\ordjac(w)} T^{u \cdot w}.
\]

\end{enumerate}
\end{theorem}

\begin{remark}
The classes above are well defined by Propositions \ref*{initialdegenerationinvariant}, \ref*{binomialcharacterinvariant}, and \ref*{initialdegenerationXschemestructure}.
\end{remark}

Let $Z^{\hat\mu}_{X,u,k}(T) \in \sM_k^{\hat\mu}\llbracket T \rrbracket$ be the power series obtained by pushing forward each coefficient of $Z^{\hat\mu}_{X,u}(T)$ along the structure morphism of $X$, and if the origin of $\bA_k^n$ is contained in $X$, let $Z^{\hat\mu}_{X,u,0}(T) \in \sM_k^{\hat\mu}\llbracket T \rrbracket$ be the power series obtained by pulling back each coefficient of $Z^{\hat\mu}_{X,u}(T)$ along the inclusion of the origin into $X$.

\begin{corollary}
\label{tropicalzetaformulahomogeneous}
Again let $V_u$ be the subscheme of $\bG_{m,k}^n$ defined by $(x_1, \dots, x_n)^u - 1$. Suppose there exists $v \in \Z^n$ such that $u \cdot v > 0$ and such that for all $w \in \Z^n$,
\[
	\init_w U = \init_{w+v} U.
\]
Then for all $w \in \Trop(U) \cap (\Z_{\geq 0}^n \setminus \{0\})$, we have that $V_u \cap \init_w U$ is invariant under the $(\mu_{u \cdot v}, v)$-action, and there exists a function $\ordjac: \Trop(U) \cap \Z_{\geq 0}^n \to \Z$ that satisfies the following.
\begin{enumerate}[(a)]

	\item If $w = (w_1, \dots, w_n) \in \Trop(U) \cap \Z_{\geq 0}^n$ and $f_1, \dots, f_{n-d} \in k[x_1, \dots, x_n]$ are a generating set for the ideal of $X$ such that $\init_w f_1, \dots, \init_w f_{n-d} \in k[x_1^{\pm 1}, \dots, x_n^{\pm 1}]$ generate the ideal of $\init_w U$, then
\[
	\ordjac(w) = w_1 + \dots + w_n - (\trop(f_1)(w) + \dots + \trop(f_{n-d})(w)) \in \Z.
\]

	\item We have that
\[
	Z^{\hat\mu}_{X,u,k}(T) = \sum_{w \in \Trop(U) \cap (\Z_{\geq 0}^n \setminus \{0\})} [(V_u \cap \init_w U)^v_{u \cdot v}, \hat\mu] \bL^{-d-\ordjac(w)} T^{u \cdot w}.
\]

	\item If the origin of $\bA_k^n$ is contained in $X$, then
\[
	Z^{\hat\mu}_{X,u,0}(T) = \sum_{w \in \Trop(U) \cap \Z_{> 0}^n} [(V_u \cap \init_w U)^v_{u \cdot v}, \hat\mu] \bL^{-d-\ordjac(w)} T^{u \cdot w}.
\]

\end{enumerate}
\end{corollary}

\subsection{Proof of \autoref*{tropicalzetaformulahomogeneous}}

Before we prove \autoref*{initialdegenerationXschemestructure} and \autoref*{zetafunctionschon}, we will show that they imply \autoref*{tropicalzetaformulahomogeneous}.

\begin{proposition}
\label{initialformmappedtoorigin}
Let $w \in \Trop(U) \cap \Z_{\geq 0}^n$, suppose that the origin of $\bA_k^n$ is contained in $X$, and endow $\init_w U$ with the $X$-scheme structure given by \autoref*{initialdegenerationXschemestructure}. Then
\begin{enumerate}[(a)]

\item if $w \in \Z_{> 0}^n$, the fiber of $\init_w U$ over the origin of $\bA_k^n$ is equal to $\init_w U$,

\item and if $w \notin \Z_{>0}^n$, the fiber of $\init_w U$ over the origin of $\bA_k^n$ is empty.

\end{enumerate}
\end{proposition}

\begin{proof}
This is a direct consequence of the $X$-scheme structure of $\init_w U$.
\end{proof}

Using the notation in the theorem's statement, \autoref*{zetafunctionschon} implies
\[
	Z^{\hat\mu}_{X,u,k}(T) = \sum_{w \in \Trop(U) \cap (\Z_{\geq 0}^n \setminus \{0\})} [(V_u \cap \init_w U)^w_{u \cdot w}, \hat\mu] \bL^{-d-\ordjac(w)} T^{u \cdot w},
\]
and if in addition, the origin of $\bA_k^n$ is contained in $X$, \autoref*{initialformmappedtoorigin} and \autoref*{zetafunctionschon} imply
\[
	Z^{\hat\mu}_{X,u,0}(T) = \sum_{w \in \Trop(U) \cap \Z_{> 0}^n} [(V_u \cap \init_w U)^w_{u \cdot w}, \hat\mu] \bL^{-d-\ordjac(w)} T^{u \cdot w}.
\]
Thus \autoref*{tropicalzetaformulahomogeneous} follows from \autoref*{zetafunctionschon} and the following proposition.

\begin{proposition}
\label{laststeptogetzetaformulacorollary}
Suppose there exists $v \in \Z^n$ such that $u \cdot v > 0$ and such that for all $w \in \Z^n$,
\[
	\init_w U = \init_{w+v} U.
\]
Let $w \in \Trop(U) \cap (\Z_{\geq 0}^n \setminus \{0\})$, and let $V_u$ be the subscheme of $\bG_{m,k}^n$ defined by $(x_1, \dots, x_n)^u-1$. Then $V_u \cap \init_w U$ is invariant under the $(\mu_{u \cdot v}, v)$-action and
\[
	[(V_u \cap \init_w U)^w_{u \cdot w}, \hat\mu] = [(V_u \cap \init_w U)^v_{u \cdot v}, \hat\mu] \in K_0^{\hat\mu}(\Var_k).
\]
\end{proposition}

\begin{proof}
Because $u \cdot w > 0$, there exist $\ell, \ell' \in \Z_{>0}$ and $w' \in \Z^n$ such that $u \cdot w' = 0$ and
\[
	\ell w = \ell' v + w'.
\]
By \autoref*{initdegscalingcharacteraction}, $V_u \cap \init_w U$ is invariant under the $(\mu_{u \cdot \ell w}, \ell w)$-action and
\[
	[(V_u \cap \init_w U)^w_{u \cdot w}, \hat\mu] = [(V_u \cap \init_w U)^{\ell w}_{u \cdot \ell w}, \hat\mu] \in K_0^{\hat\mu}(\Var_k).
\]
By the hypotheses on $v$, we have that
\[
	\init_{w'} U = \init_{\ell w} U,
\]
so by \autoref*{initialdegenerationinvariant}, $\init_{w'}U$ is invariant under the $(\mu_{u \cdot \ell w}, \ell w)$-action. Then by \autoref*{moninitialdeg}, $\init_{w'} U$ is invariant under the $(\mu_{u \cdot \ell' v}, \ell' v)$-action, and noting that $u \cdot \ell w = u \cdot \ell' v$,
\[
	[(V_u \cap \init_{w'} U)^{\ell w}_{u \cdot \ell w}, \mu_{u \cdot \ell w}] = [(V_u \cap \init_{w'} U)^{\ell' v}_{u \cdot \ell' v}, \mu_{u \cdot \ell' v}] \in K_0^{\mu_{u \cdot \ell' v}}(\Var_k).
\]
Again by \autoref*{initdegscalingcharacteraction}, $V_u \cap \init_{w'} U$ is invariant under the $(\mu_{u \cdot v}, v)$-action and
\[
	[(V_u \cap \init_{w'} U)^{\ell' v}_{u \cdot \ell' v}, \hat\mu] = [(V_u \cap \init_{w'} U)^v_{u \cdot v}, \hat\mu] \in K_0^{\hat\mu}(\Var_k).
\]
All together, noting that $\init_w U = \init_{\ell w} U = \init_{w'} U$,
\begin{align*}
	[(V_u \cap \init_w U)^w_{u \cdot w}, \hat\mu] &= [(V_u \cap \init_w U)^{\ell w}_{u \cdot \ell w}, \hat\mu]\\
	&= [(V_u \cap \init_{w'} U)^{\ell w}_{u \cdot \ell w}, \hat\mu]\\
	&= [(V_u \cap \init_{w'} U)^{\ell' v}_{u \cdot \ell' v}, \hat\mu]\\
	&= [(V_u \cap \init_{w'} U)^v_{u \cdot v}, \hat\mu]\\
	&= [(V_u \cap \init_w U)^v_{u \cdot v}, \hat\mu].
\end{align*}
\end{proof}

\subsection{Fibers of tropicalization}

For the remainder of Section \ref*{motiviczetafunctionsforschonvarieties}, fix $\ell \in \Z_{>0}$ and endow $R$ with the $\mu_\ell$-action where each $\xi \in \mu_{\ell}$ acts on $R$ by the $\pi$-adically continuous $k$-morphism $\pi \mapsto \xi^{-1}\pi$. Let $\bA_R^n = \Spec(R[x_1, \dots, x_n])$, let $\fX = X \times_k \Spec(R) \subset \bA_R^n$, and endow $\bA_R^n$ (resp. $\fX$) with the $\mu_\ell$-action induced by the $\mu_\ell$-action on $R$ and the trivial $\mu_\ell$-action on $\bA_k^n$ (resp. $X$).

Let $A_\ell \subset \sG(\fX)$ be the subset of arcs where $(x_1, \dots, x_n)^u|_\fX$ has order $\ell$, and let $A_{\ell, 1} \subset \sG(\fX)$ be the subset of arcs where $(x_1, \dots, x_n)^u|_\fX$ has order $\ell$ and angular component 1.

Let $\trop: \sG(\fX) \to (\Z_{\geq 0} \cup \{\infty\})^n$ be the function $(\ord_{x_1|_\fX}, \dots, \ord_{x_n|_\fX})$. Any arc that tropicalizes to a point in $\Z_{\geq 0}^n$ has generic point in $U \times_k \Spec(R)$, so
\[
	\trop(\sG(\fX)) \cap \Z_{\geq 0}^n \subset \Trop(U).
\]
Also because $u \in \Z_{>0}^n$ and $\ell \neq 0$,
\[
	\trop(A_\ell) \subset \trop(\sG(\fX)) \cap (\Z_{\geq 0}^n \setminus \{0\}) \subset \Trop(U) \cap (\Z_{\geq 0}^n \setminus \{0\}).
\]
Thus
\[
	A_\ell = \bigcup_{\substack{w \in \Trop(U) \cap (\Z_{\geq 0}^n \setminus \{0\})\\ u \cdot w = \ell}} \trop^{-1}(w).
\]
This union is disjoint, and because $u \in \Z_{> 0}^n$, it is also finite. By \autoref*{orderangularcomponentgreenbergscheme}, for each $w \in \Z_{\geq 0}^n$, we have that the fiber $\trop^{-1}(w)$ and the intersection $\trop^{-1}(w) \cap A_{\ell, 1}$ are $\mu_\ell$-invariant cylinders in $\sG(\fX)$. We have thus proved the following.

\begin{proposition}
\label{summingvolumesoffibersgivescoefficient}
We have that
\[
	\mu_{\fX}^{\mu_\ell}(A_{\ell, 1}) = \sum_{\substack{w \in \Trop(U) \cap (\Z_{\geq 0}^n \setminus \{0\})\\ u \cdot w = \ell}} \mu_{\fX}^{\mu_\ell} (\trop^{-1}(w) \cap A_{\ell, 1}),
\]
and
\[
	\mu_{\fX}(A_\ell) = \sum_{\substack{w \in \Trop(U) \cap (\Z_{\geq 0}^n \setminus \{0\})\\ u \cdot w = \ell}} \mu_{\fX}(\trop^{-1}(w)).
\]
\end{proposition}

\subsection{Morphisms for computing volumes}
\label{morphismsforcomputingvolumes}

Throughout Subsection \ref*{morphismsforcomputingvolumes}, we will fix some $w = (w_1, \dots, w_n) \in \Trop(U) \cap (\Z_{\geq 0}^n \setminus \{0\})$ such that $u \cdot w = \ell$. We will construct a smooth, pure relative dimension $d$, finite type, separated $R$-scheme $\fX^w$ with good $\mu_\ell$-action making the structure morphism equivariant, and we will construct a $\mu_\ell$-equivariant morphism $\psi_w: \fX^w \to \fX$ that will eventually be used to compute the motivic volumes of $\trop^{-1}(w) \cap A_{\ell, 1}$ and $\trop^{-1}(w)$.

Let $\bG_{m,R}^n = \Spec(R[x_1^{\pm 1}, \dots, x_n^{\pm 1}]) = \bG_{m,k}^n \times_k \Spec(R)$, and endow it with the $\mu_\ell$-action induced by the $\mu_\ell$-action on $\Spec(R)$ and the $(\mu_\ell, w)$-action on $\bG_{m,k}^n$. Let $\varphi_w: \bG_{m,R}^n \to \bA_R^n$ be the $R$-scheme morphism corresponding to the $R$-algebra morphism
\[
	\varphi_w^*: \Spec(R[x_1, \dots, x_n]) \to \Spec(R[x_1^{\pm 1}, \dots, x_n^{\pm 1}]): x_i \mapsto \pi^{w_i}x_i.
\]

\begin{proposition}
\label{varphiwequivariant}
The morphism $\varphi_w: \bG_{m,R}^n \to \bA_R^n$ is $\mu_\ell$-equivariant.
\end{proposition}

\begin{proof}
Let $\xi \in \mu_\ell$, and let $\xi_1: R[x_1^{\pm 1}, \dots, x_n^{\pm 1}] \to R[x_1^{\pm 1}, \dots, x_n^{\pm 1}]$ and $\xi_2: R[x_1, \dots, x_n] \to R[x_1, \dots, x_n]$ be its actions. We need to show that
\[
	\xi_1 \circ \varphi_w^* = \varphi_w^* \circ \xi_2.
\]
Because the structure morphisms of $\bG_{m,R}^n$ and $\bA_R^n$ are $\mu_\ell$-equivariant, it is sufficient to show that if $i \in \{1, \dots, n\}$, then
\[
	\xi_1(\varphi_w^*(x_i)) = \varphi_w^*(\xi_2(x_i)),
\]
which holds because
\begin{align*}
	\xi_1(\varphi_w^*(x_i)) &= \xi_1(\pi^{w_i} x_i)\\
	&= (\xi^{-1} \pi)^{w_i}\xi^{w_i} x_i\\
	&= \pi^{w_i} x_i\\
	&= \varphi_w^*(x_i)\\
	&= \varphi_w^*(\xi_2(x_i)).
\end{align*}
\end{proof}

Now let $\fX_\eta$ be the generic fiber of $\fX$, let $\varphi_{w,\eta}: \bG_{m,K}^n \to \bA_{K}^n$ be the base change of $\varphi_w$ to the fraction field $K$ of $R$, let $\fX^w_\eta \subset \bG_{m,K}^n$ be the pre-image of $\fX_\eta$ under $\varphi_{w,\eta}$, and let $\fX^w \subset \bG_{m,R}^n$ be the unique closed subscheme of $\bG_{m,R}^n$ that is flat over $R$ and has generic fiber $\fX^w_\eta$, see for example \cite[Section 4]{Gubler}. By construction, the generic fiber of $\fX^w$ is isomorphic to $U \times_k \Spec(K)$, and its special fiber is equal to $\init_w U \subset \bG_{m,k}^n$, which is smooth by the hypotheses on $X$. Thus $\fX^w$ is smooth and pure relative dimension $d$ over $R$. Note that by uniqueness, $\fX^w$ is equal to the closed subscheme of $\varphi_w^{-1}(\fX)$ defined by its $R$-torsion ideal. Thus we have a morphism $\psi_w: \fX^w \to \fX$ induced from $\varphi_w$ by restriction.

\begin{remark}
\label{openimmersionongenericfibers}
Note that if $\psi_{w,\eta}: \fX^w_\eta \to \fX_\eta$ is the base change of $\psi_w$ to $K$, we have that $\psi_{w,\eta}$ is isomorphic to the open immersion $U \times_k \Spec(K) \to X \times_k \Spec(K)$. In particular, $\psi_w$ induces an open immersion on generic fibers.
\end{remark}

To obtain a generating set for the ideal defining $\fX^w$ in $\bG_{m,R}^n$, we first need to prove the following lemma.

\begin{lemma}
\label{correctspecialfibergivesflat}
Let $\fY$ be a finite type $R$-scheme, and let $\fY^\flat$ be the closed subscheme of $\fY$ defined by its $R$-torsion ideal. If as closed subschemes of $\fY$, the special fiber of $\fY^\flat$ is equal to the special fiber of $\fY$, then $\fY$ is a flat $R$-scheme.
\end{lemma}

\begin{proof}
We may assume $\fY = \Spec(\fA)$ for some finite type $R$-algebra $\fA$. Let $I \subset \fA$ be the $\pi$-torsion ideal of $\fA$. Because $I$ is finitely generated, there exists $m \in \Z_{\geq 0}$ such that $\pi^mI = 0$. By the hypotheses,
\[
	I \subset \pi \fA.
\]
Let $f \in I$. Then there exists $g \in \fA$ such that $f = \pi g$. But $\pi g \in I$ implies that $g \in I$. Thus
\[
	I = \pi I = \pi^m I = 0.
\]
Therefore $\fA$ is $\pi$-torsion free, so it is flat over $R$.
\end{proof}

We can now prove the following two propositions.

\begin{proposition}
\label{grobnerflatness}
Let $f_1, \dots, f_m \in k[x_1, \dots, x_n]$ be a generating set for the ideal defining $X$ in $\bA_k^n$ such that $\init_w f_1, \dots, \init_w f_m \in k[x_1^{\pm 1}, \dots, x_n^{\pm 1}]$ form a generating set for the ideal of $\init_w U$ in $\bG_{m,k}^n$. Then
\[
	\pi^{-\trop(f_1)(w)}\varphi_w^*(f_1), \dots, \pi^{-\trop(f_m)(w)}\varphi_w^*(f_m) \in R[x_1^{\pm 1}, \dots, x_n^{\pm 1}]
\]
form a generating set for the ideal defining $\fX^w$ in $\bG_{m,R}^n$.
\end{proposition}

\begin{proof}
Let $\fY$ be the closed subscheme of $\bG_{m,R}^n$ defined by the ideal generated by
\[
	\pi^{-\trop(f_1)(w)}\varphi_w^*(f_1), \dots, \pi^{-\trop(f_m)(w)}\varphi_w^*(f_m).
\]
Then by construction, the generic fiber of $\fY$ is equal to $\fX^w_\eta$, and $\fX^w$ is equal to the closed subscheme of $\fY$ defined by its $R$-torsion ideal. The special fiber of $\fY$ is the closed subscheme of $\bG_{m,k}^n$ defined by $\init_w f_1, \dots, \init_w f_m$ and thus is equal to $\init_w U$, which is also the special fiber of $\fX^w$. Therefore by \autoref*{correctspecialfibergivesflat}, $\fY$ is flat over $R$, so $\fX^w$ is equal to $\fY$.
\end{proof}

\begin{proposition}
The closed subscheme $\fX^w \subset \bG_{m,R}^n$ is $\mu_\ell$-invariant.
\end{proposition}

\begin{proof}
By the hypotheses on $X$, we know there exist $f_1, \dots, f_{n-d} \in k[x_1, \dots, x_n]$ that generate the ideal of $X$ such that $\init_w f_1, \dots, \init_w f_{n-d}$ generate the ideal of $\init_w U$, so by \autoref*{grobnerflatness},
\[
	\pi^{-\trop(f_1)(w)}\varphi_w^*(f_1), \dots, \pi^{-\trop(f_{n-d})(w)}\varphi_w^*(f_{n-d}) \in R[x_1^{\pm 1}, \dots, x_n^{\pm 1}]
\]
generate the ideal defining $\fX^w$ in $\bG_{m,R}^n$. 

Thus it will be sufficient to show that if $f \in k[x_1, \dots, x_n]$, $\xi \in \mu_\ell$, and $\xi_1: R[x_1^{\pm 1}, \dots, x_n^{\pm 1}] \to R[x_1^{\pm 1}, \dots, x_n^{\pm 1}]$ is its action, then $\xi_1(\pi^{-\trop(f)(w)}\varphi_w^*(f))$ is in the ideal of $R[x_1^{\pm 1}, \dots, x_n^{\pm 1}]$ generated by $\pi^{-\trop(f)(w)}\varphi_w^*(f)$. Write
\[
	f = \sum_{u' \in \Z_{\geq 0}^n} a_{u'} (x_1, \dots, x_n)^{u'},
\]
where each $a_{u'} \in k$. Then
\[
	\pi^{-\trop(f)(w)}\varphi_w^*(f) = \sum_{u' \in \Z_{\geq 0}^n} \pi^{u' \cdot w - \trop(f)(w)} a_{u'} (x_1, \dots, x_n)^{u'},
\]
so
\begin{align*}
	\xi_1(\pi^{-\trop(f)(w)}\varphi_w^*(f)) &= \sum_{{u'} \in \Z_{\geq 0}^n} (\xi^{-1} \pi)^{u' \cdot w - \trop(f)(w)} a_{u'} \xi^{u' \cdot w} (x_1, \dots, x_n)^{u'}\\
	&= \xi^{\trop(f)(w)}\sum_{u' \in \Z_{\geq 0}^n} \pi^{u' \cdot w - \trop(f)(w)} a_{u'} (x_1, \dots, x_n)^{u'}\\
	&= \xi^{\trop(f)(w)} \pi^{-\trop(f)(w)} \varphi_w^*(f).
\end{align*}
Thus we are done.
\end{proof}

We now endow $\fX^w$ with the restriction of the $\mu_\ell$-action on $\bG_{m,R}^n$. Because $\fX^w$ is affine, this $\mu_\ell$-action is good, and by construction, this $\mu_\ell$-action makes the structure morphism equivariant. By \autoref*{varphiwequivariant}, we have that the morphism $\psi_w: \fX^w \to \fX$ is $\mu_\ell$-equivariant.

\begin{remark}
\label{specialfiberofXwaction}
By construction, the special fiber of $\fX^w$ with its induced $\mu_\ell$-action is equal to $(\init_w U)^w_\ell$.
\end{remark}

\subsection{Preparing for the change of variables formula}

For the remainder of Section \ref*{motiviczetafunctionsforschonvarieties}, let $V_u$ be the subscheme of $\bG_{m,k}^n$ defined by $(x_1, \dots, x_n)^u -1$, and if $w \in \Trop(U) \cap (\Z_{\geq 0}^n \setminus \{0\})$ is such that $u \cdot w = \ell$, let $\fX^w$ and $\psi_w: \fX^w \to \fX$ be as constructed in Subsection \ref*{morphismsforcomputingvolumes}.

\begin{proposition}
\label{cylinderforcomputingangularcomponent1tropicalfiber}
Let $w \in \Trop(U) \cap (\Z_{\geq 0}^n \setminus \{0\})$ be such that $u \cdot w = \ell$. Noting that $V_u \cap \init_w U \subset \init_w U = \fX^w_0$, the subset $\theta_0^{-1}(V_u \cap \init_w U) \subset \sG(\fX^w)$ is a $\mu_\ell$-invariant cylinder, and
\[
	\mu_{\fX^w}^{\mu_\ell}(\theta_0^{-1}(V_u \cap \init_w U)) = [(V_u \cap \init_w U)^w_\ell / \fX^w_0, \mu_\ell] \bL^{-d} \in \sM^{\mu_\ell}_{\fX^w_0}.
\]
\end{proposition}

\begin{proof}
By \autoref*{binomialcharacterinvariant} and \autoref*{specialfiberofXwaction}, we have that $V_u \cap \init_w U$ is a $\mu_\ell$-invariant subscheme of $\fX^w_0$, and with the restriction of this $\mu_\ell$-action, it is equal to $(V_u \cap \init_w U)^w_\ell$. The proposition then follows from the fact that the truncation morphism $\theta_0: \sG(\fX^w) \to \sG_0(\fX^w) = \fX^w_0$ is $\mu_\ell$-equivariant and the definition of the $\mu_\ell$-equivariant motivic measure.
\end{proof}

\begin{lemma}
\label{lemmaonwmapfromtorustoaffine}
Let $w  \in \Trop(U) \cap (\Z_{\geq 0}^n \setminus \{0\})$ be such that $u \cdot w = \ell$, let $\varphi_w: \bG_{m,R}^n \to \bA_R^n$ be as in Subsection \ref*{morphismsforcomputingvolumes}, and let $k'$ be an extension of $k$. Then $\varphi_w$ induces a bijection
\[
	\sG( \bG_{m,R}^n )(k') \to \{ x \in \sG(\bA_R^n)(k') \, | \, w = (\ord_{x_1}(x), \dots, \ord_{x_n}(x))\}.
\]
\end{lemma}

\begin{proof}
Let $R' = k'\llbracket \pi \rrbracket$, and let $K'$ be its field of fractions. Because $\varphi_w$ induces an open immersion on generic fibers, it induces an injection $\bG_{m,R}^n(K') \to \bA_R^n(K')$. Because $\bG^n_{m,R}$ is separated, this implies that $\varphi_w$ induces an injection $\sG(\bG_{m,R}^n)(k') \to \sG(\bA_R^n)(k')$. We thus only need to show that the image of this injection is $\{ x \in \sG(\bA_R^n)(k') \, | \, w = (\ord_{x_1}(x), \dots, \ord_{x_n}(x))\}$.

Let $y: \Spec(R') \to \bG_{m,R}^n$. Then for each $i \in \{1, \dots, n\}$, we have that $x_i(y)$ is a unit in $R'$, so by construction, 
\[
	\varphi_w(y) \in \{ x \in \sG(\bA_R^n)(k') \, | \, w = (\ord_{x_1}(x), \dots, \ord_{x_n}(x))\}.
\]

Write $w = (w_1, \dots, w_n)$, and let $x : \Spec(R') \to \bA_R^n$ be such that $\ord_{x_i}(x) = w_i$ for each $i \in \{1, \dots, n\}$. Then for each $i \in \{1, \dots, n\}$, we have that $\pi^{-w_i} x_i(x)$ is a unit in $R'$, so we may set $y: \Spec(R') \to \bG_{m,R}^n$ to be the morphism whose pullback is given by $x_i \mapsto \pi^{-w_i}x_i(x) \in R$. By construction $\varphi_w(y) = x$, and we are done.
\end{proof}

\begin{proposition}
\label{bijectionforchangeofvariables}
Let $w \in \Trop(U) \cap (\Z_{\geq 0}^n \setminus \{0\})$ be such that $u \cdot w = \ell$. Then $\psi_w: \fX^w \to \fX$ induces bijections $\sG(\fX^w)(k') \to \trop^{-1}(w)(k')$ and $\theta_0^{-1}(V_u \cap \init_w U)(k') \to (\trop^{-1}(w) \cap A_{\ell, 1})(k')$ for all extensions $k'$ of $k$.
\end{proposition}

\begin{proof}
Fix an extension $k'$ of $k$. Because $\psi_w$ induces an open immersion on generic fibers and because $\fX^w$ is separated, we have that $\varphi_w$ induces an injection from $\sG(\fX^w)(k')$ to $\sG(\fX)(k')$. Thus we need to show that the image of $\sG(\fX^w)(k')$ is $\trop^{-1}(w)(k')$ and that the image of $\theta_0^{-1}(V_u \cap \init_w U)(k')$ is $(\trop^{-1}(w) \cap A_{\ell, 1})(k')$.

Let $y \in \sG(\fX^w)(k') \subset \sG(\bG_{m,R}^n)(k')$. By \autoref*{lemmaonwmapfromtorustoaffine}, $\psi_w(y) \in \trop^{-1}(w)(k')$. Let $x \in \trop^{-1}(w)(k') \subset \{ x' \in \sG(\bA_R^n)(k') \, | \, w = (\ord_{x_1}(x'), \dots, \ord_{x_n}(x'))\}$. By \autoref*{lemmaonwmapfromtorustoaffine}, $x$ is in the image of $\varphi_w$, where $\varphi_w$ is as in Subsection \ref*{morphismsforcomputingvolumes}. Because $\fX^w$ is the closed subscheme of $\varphi_w^{-1}(\fX)$ defined by its $R$-torsion ideal, this implies that $x$ is in the image $\psi_w$. Thus $\psi_w$ induces a bijection $\sG(\fX^w)(k') \to \trop^{-1}(w)(k')$.

Let $y \in \sG(\fX^w)(k')$. We only need to show that $\psi_w(y) \in A_{\ell, 1}(k')$ if and only if $\theta_0(y) \in (V_u \cap \init_w U)(k')$. Write $w = (w_1, \dots, w_n)$, and let $R' = k' \llbracket \pi \rrbracket$. Then
\begin{align*}
	\psi_w(y) \in A_{\ell, 1}(k') &\iff (x_1, \dots, x_n)^u(\psi_w(y)) = \pi^{\ell}(1+\pi r) \text{ for some $r \in R'$}\\
	&\iff (\pi^{w_1} x_1, \dots, \pi^{w_n} x_n)^u (y) = \pi^{u \cdot w}(1 + \pi r) \text{ for some $r \in R'$}\\
	&\iff (x_1, \dots, x_n)^u(y) = 1 + \pi r \text{ for some $r \in R'$}\\
	&\iff ((x_1, \dots, x_n)^u - 1)(\theta_0(y)) = 0\\
	&\iff \theta_0(y) \in (V_u \cap \init_w U)(k').
\end{align*}
\end{proof}

\begin{proposition}
\label{jacobiancomputation}
Let $w \in \Trop(U) \cap (\Z_{\geq 0}^n \setminus \{0\})$ be such that $u \cdot w = \ell$, and let $f_1, \dots, f_{n-d} \in k[x_1, \dots, x_n]$ be a generating set for the ideal defining $X$ in $\bA_k^n$ such that $\init_w f_1, \dots, \init_w f_{n-d} \in k[x_1^{\pm 1}, \dots, x_n^{\pm 1}]$ form a generating set for the ideal of $\init_w U$ in $\bG_{m,k}^n$. Then the jacobian ideal of $\psi_w$ is generated by
\[
	\pi^{w_1+\dots+w_n-(\trop(f_1)(w)+\dots+\trop(f_{n-d})(w))}.
\]
\end{proposition}

\begin{proof}
Let $\varphi_w: \bG_{m,R}^n \to \bA_R^n$ be as in Subsection \ref*{morphismsforcomputingvolumes}, and for any $f \in R[x_1, \dots, x_n]$, we will set
\[
	f^w = \pi^{-\trop(f)(w)}\varphi_w^*(f) \in R[x_1^{\pm 1}, \dots, x_n^{\pm 1}].
\]
Then by \autoref*{grobnerflatness}, the ideal defining $\fX^w$ is generated by $f_1^w, \dots, f_{n-d}^w$. Let $\fA_w= R[x_1^{\pm 1}, \dots, x_n^{\pm 1}]/(f_1^w, \dots, f_{n-d}^w)$ be the coordinate ring of $\fX^w$. Then we have the diagram
\begin{center}
\begin{tikzcd}
&&&0\\
\psi_w^*\Omega_{\fX/R} \arrow[r] & \Omega_{\fX^w/R} \arrow[r, twoheadrightarrow] & \Omega_{\fX^w/\fX} \arrow[ur] \arrow[r] & 0\\
\fA_w^n \arrow[u, twoheadrightarrow] \arrow[r] & \fA_w^n \arrow[u, twoheadrightarrow] \arrow[ur, twoheadrightarrow]\\
\fA_w^n \oplus \fA_w^{n-d} \arrow[u] \arrow[ur] \arrow[r] & \fA_w^{n-d} \arrow[u]
\end{tikzcd}
\end{center}
where the right vertical sequence is the presentation for the differentials module $\Omega_{\fX^w / R}$ induced by our presentation for $\fA_w$, the top horizontal sequence is the standard presentation for the relative differentials module $\Omega_{\fX^w / \fX}$, and the top left vertical arrow picks out the generators of $\psi_w^*\Omega_{\fX, R}$ induced by the coordinates of $\bA_R^n$. The diagonal sequence gives a presentation for $\Omega_{\fX^w / \fX}$ with matrix whose entries are the images in $\fA_w$ of the entries in the matrix
\[
\begin{pmatrix}
\pi^{w_1} & & & \partial f_1^w/\partial x_1 & & \partial f_{n-d}^w/ \partial x_1\\
& \ddots & & \vdots & \cdots & \vdots\\
& & \pi^{w_n} & \partial f_1^w/\partial x_n & & \partial f_{n-d}^w/\partial x_n
\end{pmatrix}
\]
For each $i \in \{1,\dots, n\}$ and $j \in \{1, \dots, n-d\}$,
\[
	\partial \varphi_w^*(f_j)/ \partial x_i = \sum_{i'=1}^n (\varphi_w^*(\partial f_j / \partial x_{i'})) (\partial \varphi_w^*(x_{i'}) /\partial x_i) = \pi^{w_i} \varphi_w^*(\partial f_j / \partial x_i),
\]
so
\[
	\partial f_j^w/\partial x_i = \pi^{w_i - \trop(f_j)(w)} \varphi_w^*( \partial f_j / \partial x_i).
\]
For each $m \in \{0, 1, \dots, n-d\}$ and size $m$ subsets $I \subset \{1, \dots, n\}$ and $J \subset \{1, \dots, n-d\}$, let $\Delta_I^J$ be the determinant of the size $m$ minor of the matrix $(\partial f_j / \partial x_i)_{i,j}$ given by rows in $I$ and columns in $J$. Then the jacobian ideal of $\psi_w$ is generated by the images in $\fA_w$ of
\[
	\{\pi^{w_1 + \dots + w_n - \sum_{j \in J}\trop(f_j)(w)} \varphi_w^* \Delta_I^J\}_{m, I, J}.
\]
Because $\fX$ is smooth and pure relative dimension $d$ over $R$, the unit ideal of $\fX$ is generated by the images in $R[x_1, \dots, x_n]/(f_1, \dots, f_{n-d})$ of
\[
	\{\Delta_I^{\{1, \dots, n-d\}} \, | \, \text{$I \subset \{1, \dots, n\}$ has size $n-d$}\},
\]
so the jacobian ideal of $\psi_w$ contains
\[
	\pi^{w_1+\dots+w_n-(\trop(f_1)(w)+\dots+\trop(f_{n-d})(w))}.
\]
Because $w \in \Z_{\geq 0}^n$ and each $f_j \in k[x_1, \dots, x_n]$, we have that each $\trop(f_j)(w) \geq 0$. Therefore the jacobian ideal of $\psi_w$ is in fact generated by
\[
	\pi^{w_1+\dots+w_n-(\trop(f_1)(w)+\dots+\trop(f_{n-d})(w))}.
\]
\end{proof}

\subsection{Proofs of \autoref*{initialdegenerationXschemestructure} and \autoref*{zetafunctionschon}}

We prove \autoref*{initialdegenerationXschemestructure}.

\begin{proof}[Proof of \autoref*{initialdegenerationXschemestructure}]
This is clear if $w = 0$, so we may assume that $w \in \Trop(U) \cap (\Z_{\geq 0}^n \setminus \{0\})$ is such that $u \cdot w = \ell$. In this case, the proposition follows by \autoref*{specialfiberofXwaction} and by considering the special fiber of $\psi_w: \fX^w \to \fX$.
\end{proof}

Now set $\ordjac: \Trop(U) \cap (\Z^n_{\geq 0} \setminus \{0\}) \to \Z: w \mapsto \ordjac_{\psi_w}(y)$ for any $y \in \sG(\fX^w)$, noting that by \autoref*{jacobiancomputation}, this does not depend on the choice of $y$. Also set $\ordjac(0) = 0$.

\begin{proposition}
\label{computingvolumeoffiberoftropicalization}
Let $w \in \Trop(U) \cap (\Z_{\geq 0}^n \setminus \{0\})$ be such that $u \cdot w = \ell$. Then
\[
	\mu_{\fX}(\trop^{-1}(w)) = [\init_w U / X] \bL^{-d - \ordjac(w)},
\]
and
\[
	\mu_{\fX}^{\mu_\ell}(\trop^{-1}(w) \cap A_{\ell, 1}) = [(V_u \cap \init_w U)^w_{u \cdot w}/X, \mu_\ell]\bL^{-d-\ordjac(w)}.
\]
\end{proposition}

\begin{proof}
By \autoref*{openimmersionongenericfibers} and Propositions \ref*{cylinderforcomputingangularcomponent1tropicalfiber} and \ref*{bijectionforchangeofvariables}, the proposition follows from the (equivariant) motivic change of variables formula applied to $\psi_w$.
\end{proof}

Because $u \in \Z_{>0}^n$, the monomial $(x_1, \dots, x_n)^u$ vanishes on all of $X \setminus U$. Thus the constant term of $Z_{X,u}^\naive(T)$ is equal to
\[
	[U / X]\bL^{-d} = [\init_0 U / X] \bL^{-d - \ordjac(w)}.
\]
Therefore, \autoref*{zetafunctionschon} follows from \autoref*{jacobiancomputation} and the next proposition.

\begin{proposition}
The coefficient of $T^\ell$ in $Z_{X,u}^{\hat\mu}(T)$ is equal to
\[
	\sum_{\substack{w \in \Trop(U) \cap (\Z_{\geq 0}^n \setminus \{0\})\\ u \cdot w = \ell}} [(V_u \cap \init_w U)^w_{u \cdot w} / X, \hat\mu] \bL^{-d - \ordjac(w)},
\]
and the coefficient of $T^\ell$ in $Z_{X,u}^\naive(T)$ is equal to
\[
	\sum_{\substack{w \in \Trop(U) \cap (\Z_{\geq 0}^n \setminus \{0\})\\ u \cdot w = \ell}} [\init_w U / X] \bL^{-d - \ordjac(w)}.
\]
\end{proposition}

\begin{proof}
This follows from Propositions \ref*{coefficientDLfromvolume}, \ref*{coefficientIgfromvolume}, \ref*{summingvolumesoffibersgivescoefficient}, and \ref*{computingvolumeoffiberoftropicalization}.
\end{proof}

\section{Motivic zeta functions of hyperplane arrangements}
\label{motiviczetafunctionsprovingmainresultsection}

Let $d, n \in \Z_{>0}$, let $\cM$ be a rank $d$ loop-free matroid on $\{1, \dots, n\}$, and let $\cA \in \Gr_\cM(k)$. We will prove Theorems \ref*{hyperplanearrangementDLzetaformula} and \ref*{hyperplanearrangementIgusazetaformula}. Throughout this section, we will be using the notation defined in Section \ref*{linearsubspacesandmatroidssubsection}.

\begin{lemma}
\label{grobnerbasislinearsubspaceordjac}
Let $w=(w_1, \dots, w_n) \in \,\R^n$. Then there exist $f_1, \dots, f_{n-d} \in k[x_1, \dots, x_n]$ that generate the ideal of $X_\cA$ in $\bA_k^n$ such that the ideal of $\init_w U_\cA$ in $\bG_{m,k}^n$ is generated by $\init_w f_1, \dots, \init_w f_{n-d} \in k[x_1^{\pm 1}, \dots, x_n^{\pm 1}]$ and such that
\[
	w_1 + \dots + w_n - (\trop(f_1)(w) + \dots + \trop(f_{n-d})(w)) = \wt_{\cM}(w).
\]
\end{lemma}

\begin{proof}
Fix some $B \in \cB(\cM_w)$. Then by \autoref*{gbasisforlinearsubspace},
\[
	\{L_{C(\cM, i,B)}^\cA \, | \, i \in \{1, \dots, n\} \setminus B\} \subset k[x_1, \dots, x_n]
\]
generate the ideal of $X_\cA$ in $\bA_k^n$ and
\[
	\{ \init_w L_{C(\cM, i, B)}^\cA \, | \, i \in \{1, \dots, n\} \setminus B\} \subset k[x_1^{\pm}, \dots, x_n^{\pm}]
\]
generate the ideal of $\init_w U_\cA$ in $\bG_{m,k}^n$. By \autoref*{wmaximalextrahasminimalweight},
\[
	\trop(L_{C(\cM,i,B)}^\cA)(w) = w_i
\]
for each $i \in \{1, \dots, n\} \setminus B$. Thus
\[
	w_1 + \dots + w_n - \sum_{i \in \{1, \dots, n\} \setminus B} \trop(L_{C(\cM,i,B)}^\cA)(w) = \sum_{i \in B} w_i = \wt_\cM(w).
\]
\end{proof}

We now prove \autoref*{hyperplanearrangementDLzetaformula}.

\begin{proposition}
We have that
\[
	Z^{\hat\mu}_{\cA, k}(T) = \sum_{w \in \Trop(\cM) \cap (\Z_{\geq 0}^n \setminus \{0\})} [F_{\cA_w}, \hat\mu] \bL^{-d-\wt_\cM(w)}(T, \dots, T)^w \in \sM_k^{\hat\mu}\llbracket T \rrbracket,
\]
and
\[
	Z^{\hat\mu}_{\cA,0}(T) = \sum_{w \in \Trop(\cM) \cap \Z_{> 0}^n} [F_{\cA_w}, \hat\mu] \bL^{-d-\wt_\cM(w)}(T, \dots, T)^w \in \sM_k^{\hat\mu}\llbracket T \rrbracket.
\]
\end{proposition}

\begin{proof}
By setting $X = X_\cA$, $u = (1, \dots, 1)$, and $v = (1, \dots, 1)$, the proposition follows directly from \autoref*{tropicalzetaformulahomogeneous} and \autoref*{grobnerbasislinearsubspaceordjac}.
\end{proof}

We end this section by proving \autoref*{hyperplanearrangementIgusazetaformula}.

\begin{proposition}
We have that
\[
	Z^{\naive}_{\cA,k}(T) = \sum_{w \in \Trop(\cM) \cap \Z_{\geq 0}^n} \chi_{\cM_w}(\bL) \bL^{-d-\wt_\cM(w)}(T, \dots, T)^w \in \sM_k\llbracket T \rrbracket,
\]
and 
\[
	Z^{\naive}_{\cA,0}(T) = \sum_{w \in \Trop(\cM) \cap \Z_{> 0}^n} \chi_{\cM_w}(\bL) \bL^{-d-\wt_\cM(w)}(T, \dots, T)^w \in \sM_k\llbracket T \rrbracket.
\]
\end{proposition}

\begin{proof}
By setting $X = X_\cA$ and $u = (1, \dots, 1)$, the proposition follows from \autoref*{zetafunctionschon}, \autoref*{initialformmappedtoorigin}, \autoref*{grobnerbasislinearsubspaceordjac}, and the fact that for each $w \in \Trop(\cM)$, the class $[U_{\cA_w}] \in \sM_k$ is equal to $\chi_{\cM_w}(\bL)$.
\end{proof}

\bibliographystyle{alpha}
\bibliography{MZFHA}

\end{document}